\documentclass[a4paper,11pt]{article}

\makeatletter

\@addtoreset{paragraph}{subsection}
\makeatother

\usepackage{pstricks,pstricks-add,pst-math,pst-xkey}
 \usepackage{amssymb,amsmath, amsthm}
\usepackage[T1]{fontenc}
\usepackage[latin1]{inputenc}
\usepackage[english]{babel}
\usepackage{amsfonts,amsbsy}
\usepackage{enumerate}
\usepackage{mathrsfs}
\usepackage{graphicx,psfrag}
\usepackage{epic}
\usepackage{stmaryrd}
\usepackage{longtable}
\usepackage
{hyperref}

 \usepackage{fancyhdr}
 
\theoremstyle{plain}
\newtheorem{proposition}{Proposition}[section]
\newtheorem{theorem}[proposition]{Theorem}
\newtheorem{lemma}[proposition]{Lemma}
\newtheorem{corollary}[proposition]{Corollary}

\theoremstyle{definition}
\newtheorem{definition}[proposition]{Definition}

\theoremstyle{remark}

\newtheorem{remark}[proposition]{Remark}

\renewenvironment{proof}{\smallskip\noindent\emph{\textbf{Proof.}}\hspace{1pt}}%
{\hspace{-5pt}{\nobreak\quad\nobreak\hfill\nobreak$\square$\vspace{8pt}%
\par}\smallskip\goodbreak}

\renewcommand{\leq}{\leqslant}
\renewcommand{\geq}{\geqslant}

\newcommand{\C}[1]{\mathscr{C}^{#1}}

\newcommand{\modulo}[1]{{\left|#1\right|}}
\newcommand{\norma}[1]{{\left\|#1\right\|}}

\newcommand{\reali}{{\mathbb{R}}}
\newcommand{\naturali}{{\mathbb{N}}}

\renewcommand{\epsilon}{\varepsilon}
\renewcommand{\phi}{\varphi}
\renewcommand{\L}[1]{{\mathbf{L}^#1}}

\newcommand{\pt}{\partial}

\renewcommand{\d}[1]{\mathinner{\mathrm{d}{#1}}}

\newcommand{\g}{{\gamma_0}}
\newcommand{\e}{{\nu}}

\makeatletter

\@addtoreset{equation}{section}
\makeatother

\usepackage{pstricks,pstricks-add,pst-math,pst-xkey}

\title{A priori estimates and analytical construction of radially symmetric solutions in the gas dynamics}

\author{Magali L\'ecureux-Mercier
}

\begin{document} 
\maketitle

\begin{abstract}
In this article we derive $\C1$-a priori estimates on the Riemann invariants of the Euler compressible equations in the case of cylindrical or spherical symmetry. These estimates allow to construct shock waves with a time of existence proportional to the distance to the origin at the initial time.

\noindent\textit{2000~Mathematics Subject Classification:} 35L60, 35Q31, 76N10.

\medskip


\end{abstract}

\section{Introduction}
We are interested in the Euler compressible system in the isentropic case:
\begin{equation}\label{eq:isentropique}
\left\{
\begin{array}{l}
\pt_t \rho + \mathrm{div}(\rho {u})=0 \,,\\
\pt_t  u + (u\cdot\nabla )u + \frac{p'(\rho)}{\rho}  \nabla \rho=0\,.  
\end{array}
\right.
\end{equation}
In the case of cylindrical ($d=2$) or spherical ($d=3$) symmetry, the system (\ref{eq:isentropique}) can be written
\begin{equation}\label{eq:source}
\left\{
\begin{array}{l}
\pt_t \rho +\pt_r (\rho u )=\frac{-(d-1)\rho u}{r}\,,\\
\pt_t u+u\pt_r u+\frac{p'(\rho)}{\rho}\pt_r \rho=0\,.
\end{array}
\right.
\end{equation}
The case $d=1$ corresponds to the one-dimensional case.
Our goal here is to construct shock wave solutions in the isentropic spherical or cylindrical case with a reasonable lower bound on the time of existence. 

Very often only the case of a perfect polytropic gas satisfying the sate law $pv=\mathfrak{R}T$ is considered. However, in numerous phenomenon such as cavitation or sonoluminescence \cite{sonolum, Evans} or considering a dusty gas \cite{dustgret,JenaSharma,dustsaff,steiner,dustvish}, it seems more adapted to consider at least a Van der Waals gas satisfying $p(v-b)=\mathfrak{R}T$. We prove results in this more general framework.

\medskip

The Euler compressible equations have already been widely studied. Concerning regular solutions, general classical criteria on hyperbolic systems  (Leray~\cite{Leray}, G\r arding \cite{Garding}, Kato \cite{Kato}) provide us local in time existence of smooth solutions for the Cauchy problem. However,  the time of existence can be very small: several results by T. C. Sideris \cite{Sideris85, Sideris97}, T. Makino, S. Ukai \& S. Kawashima \cite{MakinoUK}, J.-Y. Chemin \cite{chemin} provide explosion's criteria. We also know that some regular solutions can be global in time:  for example stationary solutions, or under some expansivity hypothesis (see T. T. Li \cite{li}, D. Serre \cite{Sr97} or M. Grassin \cite{Grass}, M. Lécureux-Mercier \cite{mercier_solreg}).

Concerning piecewise regular solutions, a result by A. Majda \cite{Majda} states  that we can associate a piecewise regular solution to a given piecewise initial data satisfying some compatibility conditions. But the time of existence of these solutions can be once again very small.

\medskip

In this paper, we construct single propagating shock waves with a long time of existence. We consider only 2-shocks. The general framework is the one of a Bethe-Weyl  gas (for example Van der Waals gas or perfect gas,  the definition   of a Bethe-Weyl fluid is given in definition \ref{def:BW}.  ) in the particular case of spherical or cylindrical symmetry. More precisely  : we consider the cylindrical or spherical case for a Bethe-Weyl fluid. We assume that at initial time we have a piecewise regular solution with a discontinuity jump at radius $R_0$ satisfying the Lax compatibility conditions. We assume furthermore that, prolongating each regular piece of this initial condition into a regular global in space function, we can find a regular solution of system (\ref{eq:source}) with this initial condition admitting a long time of existence. Then we add some expansivity  hypotheses on the Rimann-invariant. The definition of the Riemann invariants is given in (\ref{def:w}). We obtain that the time of existence of the CS/SS solution is proportional to the radius $R_0$ of the initial discontinuity.
The statement of the main result is more precisely given in Theorem \ref{thm:main}.

The strategy of the proof is as follows: we use a method that Li Ta Tsien \cite{li} employed in order to construct 1D shock waves for an isentropic gas.
This method is inspired from the scalar one-dimensional case $\pt_t u+\pt_x(f(u))=0$, in which it is possible to obtain a shock wave solution just gluing two regular solutions along a line of discontinuity satisfying the Rankine-Hugoniot shock condition 
\[
 U=\frac{f(u^+)-f(u^-)}{u^+-u^-}\,,
\]
where $u^+$ is the limit of $u$ at the discontinuity from the right and $u^-$ is the limit of $u$ at the discontinuity from the left. This provides us an ODE that the line of discontinuity has to satisfy, being defined as $\frac{\d{x}}{\d{t}}=U$. Then, we have just to check that under suitable conditions the Lax entropy conditions meaning the characteristic curves are entering the shock (see Figure \ref{fig:1D2D}) are satisfied.

\begin{figure}[ht] 

\psset{xunit=1.0cm,yunit=1.0cm, algebraic=true,dotstyle=*,dotsize=3pt 0,linewidth=0.8pt,arrowsize=3pt 2,arrowinset=0.25}
\begin{pspicture*}(-0.74,-1.12)(6.98,5.04)
\psaxes[labelFontSize=\scriptstyle,xAxis=true,yAxis=true,labels=none,Dx=1,Dy=1,ticksize=-2pt 0,subticks=2]{->}(0,0)(-0.74,-1.12)(6.98,5.04)
\parametricplot[linewidth=1.6pt]{2.031109347472741}{3.0912795849190973}{1*5.77*cos(t)+0*5.77*sin(t)+8.08|0*5.77*cos(t)+1*5.77*sin(t)+-0.29}
\parametricplot[linewidth=1.2pt,linestyle=dashed,dash=1pt 1pt]{1.9723408695631803}{2.73272213626104}{1*5.2*cos(t)+0*5.2*sin(t)+5.19|0*5.2*cos(t)+1*5.2*sin(t)+-2.07}
\parametricplot[linewidth=1.2pt,linestyle=dashed,dash=1pt 1pt]{3.414371542275274}{4.424795547836878}{1*3.94*cos(t)+0*3.94*sin(t)+6.96|0*3.94*cos(t)+1*3.94*sin(t)+3.78}
\rput[tl](4.14,3.84){{\small shock}}
\rput[tl](4.06,1.6){{\small characteristic}}
\rput[tl](6.48,-0.1){$r$}
\rput[tl](-0.3,4.62){$t$}
\rput[tl](-0.32,-0.12){0}
\end{pspicture*}
\newrgbcolor{ccttqq}{0.8 0.2 0}
\newrgbcolor{qqzztt}{0 0.6 0.2}
\newrgbcolor{ffwwww}{1 0.4 0.4}
\psset{xunit=1.0cm,yunit=1.0cm,algebraic=true,dotstyle=*,dotsize=3pt 0,linewidth=0.8pt,arrowsize=3pt 2,arrowinset=0.25}
\begin{pspicture*}(-0.62,-1.08)(8.22,6.28)
\psaxes[labelFontSize=\scriptstyle,xAxis=true,yAxis=true,labels=none,Dx=1,Dy=1,ticksize=-2pt 0,subticks=2]{->}(0,0)(-0.74,-1.12)(6.98,5.04)
\pspolygon[linestyle=none,fillstyle=solid,fillcolor=ffwwww,opacity=0.25](2.23,0.71)(1.8,1.94)(1.54,3.12)(1.41,4.38)(1.4,4.7)(2,5)(2.96,5.32)(4.1,5.3)(5.46,5.14)(6.78,4.54)(6.44,4.46)(5.8,4.19)(5.32,3.92)(4.68,3.48)(4.33,3.17)(3.78,2.58)(3.39,2.03)(2.98,1.28)(2.71,0.6)(2.54,0)
\parametricplot[linewidth=1.6pt]{1.8712625613454594}{2.938048166068375}{1*6.03*cos(t)+0*6.03*sin(t)+8.46|0*6.03*cos(t)+1*6.03*sin(t)+-1.22}
\parametricplot[linewidth=1.2pt,linestyle=dashed,dash=2pt 2pt,linecolor=ccttqq]{3.2623028966640493}{3.579423475488481}{1*8.96*cos(t)+0*8.96*sin(t)+12.8|0*8.96*cos(t)+1*8.96*sin(t)+3.8}
\parametricplot[linewidth=1.2pt,linestyle=dashed,dash=1pt 2pt 3pt 2pt ,linecolor=qqzztt]{3.8716417996093804}{4.2482772752137565}{1*11.97*cos(t)+0*11.97*sin(t)+12.82|0*11.97*cos(t)+1*11.97*sin(t)+10.7}
\parametricplot[linewidth=1.2pt,linestyle=dashed,dash=2pt 2pt,linecolor=ccttqq]{1.8257925498382732}{2.6098944901921666}{1*5.9*cos(t)+0*5.9*sin(t)+5.39|0*5.9*cos(t)+1*5.9*sin(t)+-2.99}
\parametricplot[linewidth=1.2pt,linestyle=dashed,dash=1pt 2pt 3pt 2pt ,linecolor=qqzztt]{3.151262932604429}{3.6119394248259447}{1*10.69*cos(t)+0*10.69*sin(t)+12.09|0*10.69*cos(t)+1*10.69*sin(t)+4.84}
\parametricplot[linewidth=1.2pt,linestyle=dashed,dash=1pt 2pt 3pt 2pt ,linecolor=qqzztt]{3.0626184080623915}{3.580273898702151}{1*5.24*cos(t)+0*5.24*sin(t)+8.65|0*5.24*cos(t)+1*5.24*sin(t)+4.95}
\rput[tl](5.62,3.9){shock}
\rput[tl](5.48,1.72){{\small 1-characteristic}}
\rput[tl](3.66,0.64){{\small 2-characteristic}}
\rput[tl](0.38,3.8){{\small 1-characteristic}}
\rput[tl](0.12,0.88){{\small 2-characteristic}}
\rput[lt](2.46,6.12){\parbox{3 cm}{{\small domain of influence of the shock}}}
\rput[tl](6.48,-0.1){$r$}
\rput[tl](-0.32,5.54){$t$}
\rput[tl](-0.34,-0.16){0}
\end{pspicture*}

\caption{Shock curve and characteristics in the scalar case (left) and in the case of a system of two equations (right). the domain of uncertainty is in gray on the right picture.}\label{fig:1D2D}

\end{figure}
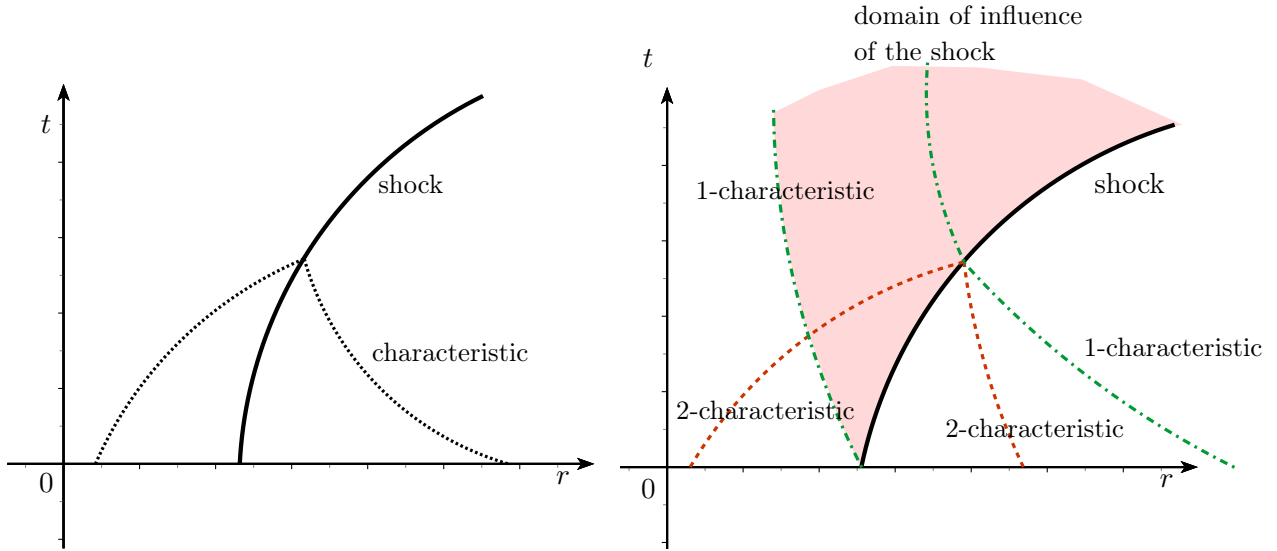

We want to apply this strategy to the  isentropic spherical Euler equation. However,   in the case of a system of two equations in one space dimension, the strategy above is no longer possible since now the Rankine-Hugoniot conditions provide us two equations. Consequently, we have not only an equation for the behavior of the shock, but also a compatibility condition along the shock. Graphically, we see that, for a 2-shock, if the 2-characteristics are still entering the shock, the 1-characteristics are entering from the right and exiting from the left (see Figure \ref{fig:1D2D}), which brings us an uncertainty zone between  the shock and the 1-characteristic exiting from the foot initial position of the shock to the left. In order to construct a shock wave solution from two regular solutions, we have to study the existence of a smooth solution in this angular domain between the shock and the 1-characteristic. Li Ta Tsien \cite{li} already studied this subject in the one-dimensional case. In particular, he obtained local in time existence of a smooth solution for the angular problem.

More precisely,  we first have to estimate the time of existence of a smooth solutions in an angular domain whose boundaries are chosen in a way to be a 1-characteristic on the left and a shock on the right (see Figure \ref{fig:1D2D}) ; the time of existence is then obtained by deriving $\C1$ estimates on the solution of this boundary problem.  Finally, thanks to a priori estimates on the Riemann invariant, we obtain a lower bound on the time of existence of these solutions, proportional to the position at initial time of the discontinuity.


\medskip

This paper is structured as followed: in Section \ref{sec:morethermo}, we describe the thermodynamical quantities and their properties. In Section \ref{sec:C0}, we compute a priori estimates in $\C0$ along the characteristics. In Section \ref{sec:C1}, we compute a priori estimates in $\C1$ along the characteristics. In Section \ref{sec:constr}, we introduce the shock conditions, the angular problem and we prove the main results of this article: in Proposition \ref{prop:constr} we give an estimate of the time of existence of a smooth solution in an angular domain and in Theorem \ref{thm:main} we finally construct a shock wave.
In Sections \ref{sec:ode} and \ref{sec:vdw}, we finally give some details about a useful lemma on ODEs and about explicit computations in the cases of a perfect gas and of a Van der Waals gas.

\section{Thermodynamics}\label{sec:morethermo}

\subsection{Fundamental relationships}

\begin{definition}
We consider a fluid, whose internal energy is a regular  function of its specific volume\footnote{specific is a synonym of  massic} $v=1/\rho$ and of its specific entropy $s$.  We say that the gas is entitled with a \emph{complete} state law, or energy law $e=e(v,s)$. 
\end{definition}

For a gas entitled with a complete state law, the fundamental thermodynamic principle is then
\begin{equation}\label{eq:fond}
\mathrm{d}e=-p\mathrm{d}v +T\mathrm{d}s\,
\end{equation}
where $p$ is the pressure and $T$ the temperature of the gas. Consequently, the pressure $p$ and the temperature $T$ can be defined as
\begin{align}
p&=-\left.\frac{\pt e}{\pt v}\right|_s\,,&
T&=\left.\frac{\pt e}{\pt s}\right|_v\,, \label{eq:pT}
\end{align}
where the notation $|$ precises the variable maintained constant in the partial derivation.

The higher order derivatives of $e$ have also an important role; we introduce the following adimensional quantities:
\begin{align}
\gamma&=-\dfrac{v}{p}\left.\dfrac{\pt p}{\pt v}\right|_s \,, &
\Gamma&=-\frac{v}{T}\left.\frac{\pt T}{\pt v}\right|_s\,, &
\delta&=\frac{pv}{T^2}\left.\frac{\pt T}{\pt s}\right|_v\,, &
\mathscr{G}&=-\frac{v}{2}\frac{\left.\frac{\pt^3 e}{\pt v^3}\right|_s}{ \left. \frac{\pt ^2 e}{\pt v^2} \right|_s}\,.
\label{eq:scrG2}
\end{align}
The  coefficient $\gamma$ is called the \emph{adiabatic exponent},  and  $\Gamma$ is the  \emph{Gr\"uneisen coefficient}.
The quantities $\gamma,\delta, \Gamma$ and $\mathscr{G}$ characterise the geometrical properties of the isentropic curves  in the $(v,p)$ plane (see \cite{MenikoffPlohr}). They can be expressed in function of $e$ through the relationships:
\begin{align*}
\gamma&= \frac{v}{p}\frac{\pt^2 e}{\pt v^2}\,,
&\Gamma&= -\frac{v}{T}\frac{\pt^2 e}{\pt s\pt v}\,, &
 \delta&=\frac{pv}{T^2}\frac{\pt^2 e}{\pt s^2}\,.
\end{align*}
 
We also introduce the \emph{calorific capacity at constant volume} $c_v$   and the \emph{calorific capacity at constant pressure} $c_p$ by
\begin{align}
c_v&=\left.\dfrac{\partial e}{\pt T}\right|_v =\frac{T}{\left.\frac{\pt^2 e}{\pt s^2}\right|_v}\,, &
c_p&=T\left.\frac{\pt s}{\pt T}\right|_p\,. \label{eq:cp}
\end{align}
These two quantities are linked with  $\frac{p v}{T}$ and with $\gamma$, $\delta$, $\Gamma$ through
\begin{equation}\label{eq:deltacv}
 \delta c_v=\frac{pv}{T}\,,\qquad\qquad c_p=\frac{pv}{T}\frac{\gamma}{\gamma \delta -\Gamma^2}\,.
\end{equation}

The quantity $\gamma_*=\frac{c_p}{c_v}$ can besides be expressed as  $\gamma_*=\frac{\gamma \delta}{\gamma\delta -\Gamma^2}$.  It is not equal to   $\gamma$  in the general case, but for an ideal gas we have $\delta=\Gamma=\gamma-1$, and consequently $\gamma_*=\gamma$.

\subsection{Thermodynamical constraints.}\label{sec:constraint}
It is very natural to assume that the massic volume $v$ is positive.  We assume furthermore  that the pressure $p$ and the temperature  $T$ are positive, which imposes that $e$ is a function increasing in $T$ and decreasing in $v$.

A classical thermodynamical hypothesis requires furthermore $e$ to be a convex function of $s$ and $v$, which means:
\begin{align*}
 \gamma\delta-\Gamma^2&\geq 0\,, &\delta&\geq 0 \,,& \gamma&\geq 0\,.
\end{align*}
Furthermore, we require usually $\Gamma>0$ and $\mathscr{G}>0$.
 The  condition $\Gamma>0$  is not thermodynamically required but is satisfied for many gases and ensures that the isentropes do not cross each other in the $(v,p)$ plan. 
The condition $\mathscr{G}>0$ means that the isentropes are strictly convex in the $(v,p)$ plan.

\begin{definition}\label{def:BW}
We call \emph{Bethe-Weyl fluid} any fluid endowed with a complete state law $e$ bounded  below such that
\begin{itemize}
\item  the pressure and the temperature defined by (\ref{eq:pT}) are positive,
\item  the coefficients $\gamma, \delta, \Gamma$ and $\mathscr{G}$ defined by  (\ref{eq:scrG2}) satisfy :
\begin{equation}\label{eq:BW}
\gamma>0\,,\qquad\gamma\delta\geq \Gamma^2\,,\qquad \Gamma>0\,,\qquad \mathscr{G}>0\,,
\end{equation}
\item there exists a maximal density  $\rho_{max}\in ]0,+\infty]$ such that  $\displaystyle \lim_{\rho\to\rho_{max}}p(\rho,s )=+\infty$.
\end{itemize}
\end{definition}

The condition $\gamma\geq 0$ means that $p$ increases with the density $\rho= 1/v$, which allows us to define the \emph{adiabatic sound speed} by 
\begin{equation}\label{eq:c2}
c= \sqrt{\left.\frac{\pt p}{\pt \rho}\right|_s} =\sqrt{\gamma\frac{p}{\rho}}\,.
\end{equation}
Then, we can check that $\mathscr{G}$   can be expressed in function of $\rho$ and $c$ through the expression 
\[
\displaystyle \mathscr{G}=\frac{1}{c}\left. \frac{\pt(\rho c)}{\pt \rho}\right|_s\,.
\]

\subsection{Van der Waals Gas}
\begin{definition}
A gas is said to follow the \emph{Van der Waals} law, if there exists a constant $\mathfrak{R}$ such that it satisfies the following pressure law:
\begin{equation}\label{eq:VdW2}
p\left(v- b \right)=\mathfrak{R}T\,,
\end{equation}
where $v$ is the massic volume and $b$ is the \emph{covolume}, representing the compressibility limit of the fluid, due to the volume of the molecules. The constant $\mathfrak{R}= 8.314 \,\mathrm{J. K^{-1}. mol^{-1}}$ is called the perfect gas constant.

In the case $b=0$, we obtain the \emph{perfect gas} law.
\end{definition}
The state law (\ref{eq:VdW2})  is a particular case  of the state law $p=\frac{\mathfrak{R}T}{V-b}-\frac{a}{V^2}$. This last law is not considered here as it authorises the change of phase when $T$ goes under a threshold $T_c=\frac{8a}{27b\mathfrak{R}}$ (see \cite{RolWi}). 

\medskip

In order to obtain expressions for the quantities $e$, $p$, $\gamma$, $\Gamma$... in function of $v$ and $s$, we use the fundamental relationship 
(\ref{eq:fond}).
 This equation gives us  the PDE: $\pt_v e+\frac{\mathfrak{R}}{v-b}\pt_s e=0$. 
Thus, we introduce new variables  $w=(v-b)^{-\mathfrak{R}}$, $\sigma=(v-b)^{-\mathfrak{R}}\exp(s)$ and $\hat e(w,\sigma)=e(v,s)$.  We obtain $\pt_w\hat e=0$, so that  $e=\mathcal{E}((v-b)^{-\mathfrak{R}}\exp(s))$ for any regular function  $\mathcal{E}$. 

If we assume furthermore that $c_v$ is  constant, thanks to the definition of $c_v$ and (\ref{eq:fond}), we get that $\left.\frac{\pt^2 e}{\pt s^2}\right|_v=\frac{1}{c_v}\left.\frac{\pt e}{\pt s}\right|_v$, hence $\sigma \mathcal{E}''=(\frac{1}{c_v}-1)\mathcal{E}'$ and  $\mathcal{E}(\sigma)=C\sigma^{1/c_v}$ which leads to:
\begin{align*}
e&=(v-b)^{-\frac{\mathfrak{R}}{c_v}}\exp(\frac{s}{c_v})\,,& p&= \frac{\mathfrak{R}}{c_v}\frac{e}{v-b}\,.
\end{align*}
After some computations we finally  obtain
\begin{align}
\gamma &=\gamma_0\frac{v}{v-b},& \Gamma=\delta&=(\gamma_0-1)\frac{v}{v-b},&\mathscr{G}&=\frac{\gamma_0+1}{2}\frac{v}{v-b}\,, \nonumber
\end{align}
where
\begin{equation}\label{eq:g}
\g :=\frac{\mathfrak{R}}{c_v}+1\,.
\end{equation}
The conditions of Section \ref{sec:constraint}  are then satisfied for $\gamma_0>1$.

In the following, we consider a general Bethe-Weyl fluid satisfying $1<\mathscr{G}<2$. 
Another general assumption is that the application $(\rho \mapsto \frac{c(\rho)}{\rho})$ is integrable in 0. 

In particular, we can consider a Van der Waals fluid with constant  and strictly positive calorific capacity $c_v$:
\begin{equation}\label{eq:gpos}
c_v>0\,,
\end{equation}
which implies $\g >1$  and $\mathscr{G}>1$.

We can also consider the standard $p$-system : $p=\rho^\gamma$ with $\gamma>1$.

\section{A priori $\C0$ estimates along the characteristics}\label{sec:C0}
Let us remind that  in this paper we consider cylindrical or spherical Euler equation in the isentropical case.

First we want to obtain $\C0$ estimates on regular solutions of (\ref{eq:source}), $\rho$ and $u$. To do that, we use the Riemann invariants of the system and we make computations along the characteristics. 

\subsection{Change of variables}
\begin{lemma} We assume that $(\rho \mapsto \frac{c(\rho)}{\rho})$ is integrable in 0.
Let us introduce the Riemann invariants
\begin{align}
w_1&=u-H(\rho)\,,& w_2&=u+H(\rho)\,,\label{def:w}
\end{align}
where $ \displaystyle H(\rho)$ is a primitive of $\rho\mapsto  \frac{c(\rho)}{\rho}$ vanishing in 0. Then, the system (\ref{eq:source}) reduces to 
\begin{equation}\label{eq:red_source}
\left\{
\begin{array}{l}
\pt_t w_1 +\lambda_1(w) \pt_r w_1=f(r,w),\\
\pt_t w_2+\lambda_2 (w) \pt_r w_2=-f(r,w),
\end{array}
\right.
\end{equation}
where $f(r,w)=\frac{(d-1)uc}{r}$, $\lambda_1=u-c(\rho)$, $ \lambda_2=u+c(\rho)$. 
\end{lemma}

\begin{proof}
Direct computation.
\end{proof}

We prove below some properties of the new unknown $H, w_1, w_2$, which will be useful in the proof of the main theorem.
\begin{lemma}\label{lem:cH} Let us consider a Bethe-Weyl gas.
We  assume that $1< \mathscr{G}< 2$ and that $(\rho \mapsto \frac{c(\rho)}{\rho})$ is integrable in 0. Then we have $H\geq c$ and in particular $u-H\leq u-c$, that is to say $w_1\leq \lambda_1$.
  
In particular,   $w_1=u-H\geq 0$ implies $\lambda_1=u-c\geq 0$. 
\end{lemma}

\begin{remark}
In the case of a perfect polytropic gas, we have $\mathscr{G}=\frac{\g+1}{2}$ and the condition $1<\mathscr{G}< 2$ is equivalent to $1< \g<3$, which is a natural hypothesis on $\g$.
\end{remark}

\begin{proof}
Note first that $c'(\rho)=H'(\rho)(\mathscr{G}-1)\geq 0$. Then $H'(\rho)-c'(\rho)=H'(\rho)(2-\mathscr{G})=\frac{c(\rho)}{\rho}(2-\mathscr{G})\geq 0$. Hence, integrating on $[0, \rho]$, we obtain $H(\rho)\geq c(\rho)$.
\end{proof}

\begin{lemma}\label{lem:cH2}
Let $\rho^+> 0$. Let us define, for $\rho\geq \rho^+$,
\begin{equation}\label{eq:F}
F(\rho, \rho^+):=\left(p-p^+\right)\left(\frac{1}{\rho^+}-\frac{1}{\rho}\right)\,.
\end{equation}
Then, for all $\rho\geq \rho^+$, we have $H(\rho)-H(\rho^+) \leq \sqrt{F(\rho, \rho^+)}$.
\end{lemma}

\begin{proof}
Let $\rho\geq \rho^+$. Let us derivate $\sqrt{F(\rho, \rho^+)}-H(\rho)$ with respect to $\rho$. We obtain
\begin{align*}
\frac{\d{}}{\d{\rho}}(\sqrt{F(\rho, \rho^+)}-H(\rho))&=\frac{1}{2\sqrt{F(\rho, \rho^+)}}\left(c^2 (\frac{1}{\rho^+}-\frac{1}{\rho}) +\frac{1}{\rho^2}(p-p^+) \right)-\frac{c}{\rho}\\
&=\frac{c}{2\rho \sqrt{F(\rho, \rho^+)}}\left(c\rho (\frac{1}{\rho^+}-\frac{1}{\rho}) +\frac{1}{c\rho}(p-p^+) -2\sqrt{F(\rho, \rho^+)}\right)\\
&=\frac{c}{2\rho\sqrt{F(\rho, \rho^+)}}\left(\sqrt{c\rho (\frac{1}{\rho^+}-\frac{1}{\rho})} -\sqrt{\frac{1}{c\rho}(p-p^+)} \right)^2\\
&\geq 0\,.
\end{align*}
Noting that $F(\rho^+, \rho^+)=0$ and integrating on $[\rho^+, \rho]$, we obtain $\sqrt{F(\rho, \rho^+)}-H(\rho)\geq -H(\rho^+)$, which is the expected result.
\end{proof}

\begin{lemma}\label{lem:w1posK}
Let $u^+, \rho^-, \rho^+\in \reali$. Let us assume $\rho^-\geq \rho^+>0$ and  define $u^-:=u^++\sqrt{F(\rho^-, \rho^+)}$, with $F$ defined as in (\ref{eq:F}). Then  $u^--H(\rho^-)\geq u^+-H(\rho^+)$, that is to say $w_1^-\geq w_1^+$. 

In particular, $w_1^+=u^+-H(\rho^+)\geq 0$ implies $w_1^-=u^--H(\rho^-)\geq 0$. 
\end{lemma}

\begin{proof}
This is a direct consequence of Lemma \ref{lem:cH2}. Indeed, the inequality $\sqrt{F(\rho^-, \rho^+)}\geq H^- -H^+$ implies
\begin{align*}
u^--H(\rho^-)=u^++\sqrt{F(\rho^-, \rho^+)}-H(\rho^-)\geq u^++H(\rho^-) - H(\rho^+) -H(\rho^-) =u^+-H(\rho^+)\,.
\end{align*}
\end{proof}

\subsection{$\C0$ estimate on $w_1$ and $w_2$}
Relying on computations along the characteristics, we now obtain estimates in $\L\infty$ for $w_1$ and $w_2$, the Riemann invariants associated to a regular solution of (\ref{eq:source}).

\begin{lemma} \label{lem:C0} 
Let us consider a Bethe-Weyl gas, satisfying $1< \mathscr{G}< 2$ and such that $(\rho \mapsto \frac{c(\rho)}{\rho})$ is integrable in 0. Let $w=(w_1, w_2)$ be a regular solution of (\ref{eq:red_source}) with a time of existence $\mathcal{T}>0$. 
Let $X_1$ and $X_2$ be the characteristics defined by
\begin{align}
\frac{\d{X_1}}{\d{t}}&=\lambda_1(w(t,X_1(t)))\,,&\frac{\d{X_2}}{\d{t}}&=\lambda_2(w(t,X_1(t)))\,,\label{def:car}\\
X_1(0)&= r_1\,,&X_2(0)&= r_2\,. \nonumber
\end{align}

Let us assume that for all $r>0$, $(w_2-w_1)(0, r)\geq 0$ and  $w_1(0,r)> 0$.
Then, we get: $X_1'> 0$, $X_2'> 0$ and, for all $t\in [0, \mathcal{T}[$ we have the estimates:
\begin{eqnarray}
&w_1(0, r_1)\leq w_1(t,X_1(t))\leq w_1(0, r_1)+\norma{w_2(0,\cdot)}_{\L\infty}^2\int_0^t \frac{d-1}{4X_1(\tau)}\d{\tau}\,, \label{eq:estw1} \\
&\frac{w_2(0,r_2)}{1+w_2(0,r_2)\int_0^t \frac{(d-1) }{4 X_2(\tau)}\d{\tau}} \leq  w_2(t,X_2(t))\leq w_2(0,r_2)\,.\label{eq:estw2}
\end{eqnarray}

Furthermore, for all $t\in [0, \mathcal{T}[$ :
\begin{align}
&w_1(0, X_1(0;t,r))\leq w_1(t,r)\leq w_2(t,r)\leq w_2(0,X_2(0;t,r))\,,
\end{align}
where, for $i\in \{1,2\}$, $X_i(0;t,r)$ designates the position at time 0 of the characteristic that satifies the initial condition $X_i(t)=r$.

\end{lemma}

\begin{proof}
First, let $t_0>0$. We introduce $\chi$ be the solution of the ODE:
\[
\frac{\d{\chi}}{\d{t}}=u(t,\chi)\,, \qquad\chi(t_0)=r_0\,.
\]
Note that the solution of this ODE exists at least in finite time. Some trouble could appear only if $\chi$ meets the line of equation $r=0$. In the same way, the characteristics $X_1$ and $X_2$ satisfying (\ref{def:car}) are defined at least in finite time. Let $r_0, r_1, r_2>0$ be fixed. Let us denote $T\in]0, \mathcal{T}[$ a time such that  $\chi$, $X_1$, $X_2$ are defined on $[0,T]$. 

\noindent\textbf{Positivity of the density:} 
Let us prove first that $\rho$  remains non-negative  because of the first equation of (\ref{eq:source}). By the definition of $\chi$, for any $t\in [0,T]$.
\[
\frac{\d{}}{\d{t}}\left( \chi^{d-1} \rho(t,\chi) \exp(\int_{t_0}^t \pt_r u(s, \chi(s)) \d{s})\right)=0\,.
\]
Consequently, thet density being non-negative at time $t_0=0$,  the densitity remains non-negative  along the lines $(t,\chi(t))$; that is to say: the inequality $w_2\geq w_1$  is satisfied for all time $t\in [0,T]$ since it is true at time $t=0$. 

\medskip


\noindent\textbf{$w_1$ is increasing along the characteristic:} 
At least on a small time interval containing 0, since $w_1(0,r_1) >0$ we have $w_1(t, X_1(t))\geq 0$. We want to prove that $w_1$ remains non-negative along the characteristic. 
Assume there is a time at which $w_1(t,X_1 (t))=0$. Let us denote $t_0\leq T$ the first time at which $w_1(t_0, X_1(t_0))=0$. Thanks to the previous result on the positivity of the density, we obtain $w_2(t, X_1(t))\geq w_1(t, X_1(t))\geq 0$ on $[0,t_0]$.  Then on $[0,t_0]$ we have $u=\frac{w_2+w_1}{2}\geq 0$ and $f(r,w)=\frac{(d-1)uc}{r}\geq 0$ which implies
\[
\frac{\d{w_1}(t,X_1(t))}{\d{t}}\geq 0\,.
\]
Integrating, we get $w_1(t_0, X_1(t_0))\geq w_1(0,r_1)>0$, which is in contradiction with the hypothesis. Finally, $w_1$ is  strictly positive for all $t\in [0,T]$ and increasing along the first characteristic. In particular, we also have $u>0$ for all $t\in [0,T]$.

\medskip
\noindent\textbf{Upper bound on $w_2$:} 
Along the second characteristic we get, since $u\geq w_1\geq 0$,
\[
\frac{\d{w_2(t,X_2(t))}}{\d{t}}=-f(X_2,w(t,X_2(t)))\leq 0\,.
\]
Integrating we obtain $w_2(t,X_2(t))\leq w_2(0,r_2)$, which provides us with an upper bound.

\medskip
\noindent\textbf{Lower bound on $w_2$:} 
Thanks to Lemma \ref{lem:cH} we know  that $c\leq H$. Hence, we obtain
\[
\frac{\d{w_2(t,X_2(t))}}{\d{t}}=\frac{-(d-1)}{X_2(t)}uc\geq \frac{-(d-1)}{X_2(t)} uH=\frac{-(d-1)}{4X_2(t)}(w_2^2-w_1^2)\geq  \frac{-(d-1)}{4X_2(t)}w_2^2\,.
\]
Consequently, we have $-\frac{1}{w_2^2(t,  X_2(t))}\frac{\d{w_2(t,X_2(t))}}{\d{t}}\leq \frac{d-1}{4X_2(t)}$ and integrating we finally obtain 
\[
\frac{1}{w_2(t, X_2(t))}\leq \frac{1}{w_2(0, r_2)} +\int_0^t \frac{d-1}{4X_2(\tau )}\d{\tau}\,.
\]
Since $w_2\geq 0$, we can invert this relation, and obtain the desired lower bound on $w_2$.

\medskip
\noindent\textbf{Upper bound on $w_1$:}
Similarly for $w_1$ we get
\[
\frac{\d{w_1(t,X_1(t))}}{\d{t}}=\frac{(d-1)}{X_1(t)}uc\leq \frac{d-1}{X_1(t)} uH=\frac{(d-1)}{4X_1(t)}(w_2^2-w_1^2)\leq  \frac{(d-1)}{4X_1(t)}\norma{w_2}_{\L\infty([0,T])}^2\,.
\]

Hence, as announced,
\[
w_1(t, X_1(t))\leq w_1(0, r_1)+\int_0^t \frac{(d-1)}{4X_1(\tau)}\norma{w_2}_{\L\infty([0,T])}^2\d{\tau }\,.
\]

\medskip
\noindent\textbf{Time of existence: }
Note now that $\frac{\d{\chi}}{\d{t}}= u(t, \chi(t))\geq 0 $ implies that this curve never meets the origin and is defined on $\reali^+$. Similarly, $\frac{\d{X_2}}{\d{t}}=(u+c)(t,X_2(t))\geq 0$ implies the 2-characteristics are going away from the origin and are defined for all $t\in \reali^+$.

Concerning the first characteristic, Lemma \ref{lem:cH} gives us that $\frac{\d{X_1}}{\d{t}}=(u-c)(t,X_1(t))\geq (u-H)(t, X_1(t))=w_1(t,X_1(t))\geq 0$. Consequently, the 1-characteristics are going away from the origin and are defined for all $t\in [0, \mathcal{T}[$.
\end{proof}

Let us modify slightly the hypotheses in order to obtain  a better result: we take off the hypothesis $w_{1,0}>0$. We obtain a similar result,  but now the time of validity is finite:
\begin{proposition}\label{prop:CO2}
Let us consider a Bethe-Weyl gas, satisfying $1< \mathscr{G}< 2$ and such that $(\rho \mapsto \frac{c(\rho)}{\rho})$ is integrable in 0. Let $T>0$, $R>0$ and let $w=(w_1, w_2)$ be a regular solution of (\ref{eq:red_source}) with a time of existence $\mathcal{T}>0$. We denote $w_1{1,0}=w_1(0,\cdot)$ and $w_{2,0}=w_2(0,\cdot)$, the initial condition of this regular solution. 
Let $T,R$ be two stricty positive real numbers and $X_1$ and $X_2$ be the characteristics defined by (\ref{def:car}), crossing in $(T,R)$. 

We assume that for all $r>0$, $(w_{2,0}-w_{1,0})( r)\geq 0$ and  $\min(w_{1,0}+w_{2,0})> 0$. Then there exists a time $T_0>0$ such that :
\begin{itemize}
\item for all $t\in [0, T_0]$,
\begin{align}
&w_{1,0} (r_1)\leq w_1(t,R)\leq w_2(t,R)\leq w_{2,0}(r_2)\,.
\end{align}
\item  for all $t\in [0,T_0]$,
\begin{eqnarray}
&w_{1,0}(r_1)\leq w_1(t, X_1(t) )\leq w_{1,0} (r_1)+\norma{w_{2,0}}_{\L\infty}^2 \int_0^t \frac{d-1}{4X_1(\tau)}\d{\tau}\,, \label{eq:estw1b} \\
&\frac{w_{2,0}(r_2)}{1+\frac{w_{2,0}(r_2)}{r_2} \frac{(d-1) t}{4 }} \leq  w_2(t,X_2(t))\leq w_{2,0}(r_2)\,.\label{eq:estw2b}
\end{eqnarray}
\end{itemize}

Furthermore, the time of existence $T_0$ is defined by $T_0=\min(\mathcal{T},+\infty)$ if $ \min(w_{1,0})\geq 0$ and, in the case $\min(w_{1,0})<0$, by
\begin{align*}
T_0&=\min\left(\mathcal{T},\frac{-4r_2}{(d-1)\min_{r_1>0}(w_{1,0}(r_1)) \min_{r_2>0}(w_{2,0}(r_2))}( \min_{r_1>0}(w_{1,0}(r_1)) + \min_{r_2>0}(w_{2,0}(r_2) ) )\right)\,,
\end{align*} 
\end{proposition}

\begin{proof}
As in the preceding proposition, we know that if the density if non-negative at the initial time, then it is non-negative for any positive time. Hence, $w_2(t,r)\geq w_1(t,r)$ for all $(t,r)\in[0, \mathcal{T}[\times \reali_+$.

 By  hypothesis,  $\inf_{r>0} u(0,r)\geq  \frac{1}{2}(\inf(w_{1,0}+w_{2,0}))>0$. Consequently,  by semi-continuity of $(t\mapsto \inf_{r>0} u(t,r))$, there exists $t_0\in ]0,\mathcal{T}[$ such that  on a small time interval $[0,t_0]$ we have $u\geq 0$. Hence   for all $t\in [0,t_0]$, $\left.\frac{\d{w_1(t,X_1(t))}}{\d{t}}\right.>0$ and $\left.\frac{\d{w_2(t,X_2(t))}}{\d{t}}\right.< 0$. Hence for $t\in [0,t_0]$, $w(t,X_1(t))\geq w_1(0,r_1)$ and $w_2(t, X_2(t))\leq w_2(0,r_2)$.

Assume furthermore that  $u(0,r)> 0$ on $[0,t_0]$, then: $\frac{\d{w_2(t, X_2(t))}}{\d{t}}\geq \frac{-(d-1)}{4 X_2(t)} w_2^2$ and $\frac{\d{w_1(t,X_1(t))}}{\d{t}} \leq  \frac{(d-1)}{4 X_1(t)}w_2^2$. Hence, the same estimates as before hold true:
\[
\frac{1}{w_2(t, X_2(t))}\leq \frac{1}{w_2(0,r_2)} +\int_0^t\frac{(d-1)}{4 X_2(\tau)}\d{\tau}
\]
Since, for $t\in [0,T_0]$, $\lambda_2=u+c\geq 0$, $X_2$ is increasing  and as $w_2\geq u$ implies $w_2\geq 0$, we can invert the relation, obtaining
\[
w_2(t,X_2(t))\geq \frac{w_2(0,r_2)}{1+ w_2(0,r_2) \frac{(d-1) t}{4 r_2}}\,.
\]

Let us use these estimate to find a lower bound for $u$. If $t\leq T_0$  then
\begin{align*}
u(t,R)=\frac{1}{2}(w_1+w_2)(t,R)\geq \frac{1}{2}\left(w_1(0,r_1) +\frac{w_{2}(0,r_2 )}{1+ w_2(0,r_2) \frac{(d-1) t}{4 r_2}} \right)\,.
\end{align*}
 If $w_{1}(0,r_1)\geq 0$, then $u$ is non-negative for all time and we recover the result of Lemma \ref{lem:C0}. If $w_1(0,r_1)<0$ and  $w_{1}(0,r_1)+w_2(0,r_2)>0$, then $u$ is non-negative if 
 \[
 t\leq \frac{-4 r_2}{(d-1)w_{1}(0,r_1) w_2(0,r_2)}(w_1(0,r_1)+w_2(0,r_2))\,,
 \] 
 which provides us a lower bound for $T_0$.
\end{proof}

\begin{remark}
In Lemma \ref{lem:C0} and Proposition \ref{prop:CO2}, we are integrating along the characteristics  on the time interval $[0,t]$. Note that we could obtain similar result integrating on any time interval $[\beta, t]$ with $\beta\in [0,t]$.
\end{remark}

\section{A priori $\C1$ estimates along the characteristics}\label{sec:C1}
We want now to obtain estimates in $\L\infty$ on $\pt_r w_1$ and $\pt_r w_2 $ where $w=(w_1,w_2)$ is a regular solution of (\ref{eq:red_source}). We  apply the same strategy as in the previous section. That is to say, we want to have a diagonal form for  the system  of equation of variables $\pt_r w_1$ and $\pt_r w_2$, obtained by derivating with respect to $r$ the system (\ref{eq:red_source}), and then make computations along the characteristics. As the obtained system in $\pt_r w_1$ and $\pt_r w_2$ is not diagonal, we have to introduce new variables $v_1$ and $v_2$ as described below. 

\subsection{Change of variable}
Note that in the following section, $e$ stands for $\exp$ and not for the internal energy. We introduce below fuctions depending on the radius $r$ and on the riemann invariant $w_1$ and $w_2$ ;  we denote $\pt_1$ for $\frac{\pt}{\pt w_1}$, $\pt_2$ for $\frac{\pt}{\pt w_2}$.

\begin{lemma}\label{lem:v1v2}
Let us define
\begin{align}
v_1&=e^{h}(\pt_r w_1 +\Phi)\,,& v_2&= e^k(\pt_r w_2 +\Psi)\,,\label{def:v}
\end{align}
where $h$, $k$, $\Phi$ and $\Psi$ are such that:
\begin{align}
\pt_2 h &=\frac{\pt_2 \lambda_1}{\lambda_1-\lambda_2}\,,&\pt_1 k&=\frac{\pt_1\lambda_2}{ \lambda_2-\lambda_1}\,, \label{eq:h}\\
\pt_2 (e^h \Phi)& =\frac{-e^h \pt_2 f}{\lambda_1-\lambda_2}\,,& \pt_1(e^k \Psi)&= \frac{e^k \pt_1 f}{\lambda_2-\lambda_1}\,.\label{eq:Phi}
\end{align}
Then $v_1$ and $v_2$ satisfy the equations
\begin{equation}\label{eq:v1v2}
\left\{
\begin{array}{lcl}
\pt_t v_1 +\lambda_1(w) \pt_r v_1 &=& a_0(r,w) v_1^2 +a_1(r,w) v_1+a_2(r,w)\,,\\
\pt_t v_2 +\lambda_2 (w) \pt_r v_2 &=& b_0(r,w) v_2^2 +b_1(r,w) v_2+b_2(r,w)\,.\\
\end{array}
\right.
\end{equation}
where
\begin{eqnarray*}
a_0(r,w)&=& - e^{-h} \pt_1\lambda_1\,,\\
a_1(r,w)&=& \pt_1 f+2\Phi \pt_1\lambda_1 +(\pt_1 h-\pt_2 h)f\,,\\
a_2(r,w)&=& e^{h}\left( \pt_r f-\Phi \pt_1 f -\Phi^2 \pt_1\lambda_1 +(\pt_1\Phi-\pt_2\Phi) f +\lambda_1 \pt_r \Phi\right) \,,\\
b_0(r,w)&=&- e^{-k} \pt_2\lambda_2\,,\\
b_1(r,w)&=&-\pt_2 f +2 \pt_2\lambda_2 \Psi +(\pt_1 k-\pt_2 k)f\,,\\
b_2(r,w)&=&e^{k}\left( -\pt_r f +\pt_2 f \Psi -\pt_2\lambda_2 \Psi^2 +\lambda_2\pt_r\Psi +(\pt_1\Psi-\pt_2\Psi) f\right)\,.
\end{eqnarray*}

\end{lemma}

\begin{remark}
Note that, for any gas law, we have $\pt_1\lambda_1=\pt_2 \lambda_2=\frac{\mathscr{G}}{2}\geq 0 $. Hence, we have  $a_0\leq 0$ and $b_0\leq 0$.
\end{remark}

\begin{proof}
On the first hand, derivating $v_1$ with respect to $t$ and $r$, we get
\begin{align*}
\pt_t v_1=&  e^{h}\left(\pt_t\pt_r w_1+\pt_1\Phi \pt_t w_1 +\pt_2 \Phi \pt_t w_2\right)+e^h (\pt_1 h \pt_t w_1 +\pt_2 h\pt_t w_2)(\pt_r w_1 +\Phi)\,,\\
\pt_r v_1=&  e^{h}(\pt^2_r w_1+\pt_1\Phi \pt_r w_1 +\pt_2 \Phi \pt_r w_2+\pt_r \Phi )
+e^h (\pt_1 h \pt_r w_1 +\pt_2 h\pt_r w_2)(\pt_r w_1 +\Phi)\,.
\end{align*}
Hence
\begin{align*}
&e^{-h}(\pt_t v_1+\lambda_1\pt_r v_1)\\
&= \pt_t\pt_r w_1 +\lambda_1\pt^2_r w_1 +\lambda_1\pt_r \Phi +\pt_1\Phi(\pt_tw_1+\lambda_1\pt_rw_1) +\pt_2\Phi (\pt_t w_2+\lambda_1\pt_r w_2)\\
&\quad+(\pt_r w_1+\Phi)(\pt_1 h (\pt_t w_1+\lambda_1\pt_r w_1) +\pt_2 h (\pt_t w_2+\lambda_1\pt_r w_2))\,.
\end{align*}
Note that  $\pt_t w_1 +\lambda_1\pt_r w_1 = f$, and $
\pt_t w_2 +\lambda_1\pt_r w_2 = -f +(\lambda_1-\lambda_2)\pt_r w_2$ so that  
\begin{align}
&e^{-h}(\pt_t v_1+\lambda_1\pt_r v_1)\nonumber\\
&= \pt_t\pt_r w_1 +\lambda_1\pt^2_r w_1 +\lambda_1\pt_r \Phi +f \pt_1\Phi +\pt_2\Phi (-f+(\lambda_1-\lambda_2 ) \pt_r w_2)\label{eq:derivee} \\
&\quad+(\pt_r w_1+\Phi)(f\pt_1 h  +\pt_2 h (-f+(\lambda_1-\lambda_2) \pt_r w_2))\,. \nonumber
\end{align}
On the other hand, derivating with respect to time the equation in $w_1$, we obtain:
\begin{align*}
&\pt_t\pt_r w_1 +\lambda_1 \pt_r^2 w_1 \\
&= \pt_r f+\pt_1 f \pt_r w_1 +\pt_2 f\pt_r w_2 -\pt_1\lambda_1 (\pt_r w_1)^2 -\pt_2\lambda_1 \pt_r w_2 \pt_r w_1\\
&=\pt_r f +\pt_1 f (e^{-h} v_1-\Phi)+\pt_2 f\pt_r w_2 -\pt_1\lambda_1 (e^{-h} v_1-\Phi)^2 -\pt_2\lambda_1 \pt_r w_2(e^{-h} v_1-\Phi)\\
&=- e^{-2h}v_1^2  \pt_1\lambda_1 +e^{-h }v_1(\pt_1 f+2 \Phi\pt_1\lambda_1) +(\pt_r f-\Phi \pt_1 f -\Phi^2\pt_1\lambda_1 )\\
&\quad +\pt_r w_2 (\pt_2 f+\Phi \pt_2 \lambda_1) -\pt_2\lambda_1 e^{-h} v_1 \pt_r w_2\,. 
\end{align*}
Replacing $\pt_t \pt r w_1+\lambda_1 \pt_1^2 w_1$ by its expression in (\ref{eq:derivee}), we get
\begin{align*}
&e^{-h}(\pt_t v_1+\lambda_1\pt_r v_1)\\
&=- e^{-2h}v_1^2  \pt_1\lambda_1+e^{-h }v_1(\pt_1 f+2 \Phi\pt_1\lambda_1)  +(\pt_r f-\Phi \pt_1 f -\Phi^2\pt_1\lambda_1 )\\
&\quad +\pt_r w_2 (\pt_2 f+\Phi \pt_2 \lambda_1) -\pt_2\lambda_1 e^{-h} v_1 \pt_r w_2 +\lambda_1\pt_r \Phi +f (\pt_1\Phi -\pt_2\Phi )+\pt_2\Phi (\lambda_1-\lambda_2 ) \pt_r w_2\\
&\quad+e^{-h } v_1 f(\pt_1 h-\pt_2 h)  +  \pt_2 h (\lambda_1-\lambda_2) e^{-h } v_1 \pt_r w_2\\
&= - e^{-2h}v_1^2  \pt_1\lambda_1 +e^{-h }v_1(\pt_1 f+2 \Phi\pt_1\lambda_1+ f(\pt_1 h-\pt_2 h) )\\
&\quad +(\pt_r f-\Phi \pt_1 f -\Phi^2\pt_1\lambda_1  +f (\pt_1\Phi -\pt_2\Phi ) +\lambda_1\pt_r \Phi )\\
&\quad +\pt_r w_2 (\pt_2 f+\Phi \pt_2 \lambda_1 +\pt_2\Phi (\lambda_1-\lambda_2 ) ) +(\pt_2 h (\lambda_1-\lambda_2) -\pt_2\lambda_1) e^{-h} v_1 \pt_r w_2 \,.
\end{align*}
With our choice for $h $ and $\Phi$, the last line above vanishes.

In the same way, we have for $v_2$
\begin{align*}
\pt_t v_2=&  e^{k}\left(\pt_t\pt_r w_2+\pt_1\Psi \pt_t w_1 +\pt_2 \Psi \pt_t w_2\right)+e^k (\pt_1 k \pt_t w_1 +\pt_2 k\pt_t w_2)(\pt_r w_2 +\Psi)\,,\\
\pt_r v_2=&  e^{k}(\pt^2_r w_2+\pt_1\Psi \pt_r w_1 +\pt_2 \Psi \pt_r w_2+\pt_r \Psi )
+e^k (\pt_1 k \pt_r w_1 +\pt_2 k\pt_r w_2)(\pt_r w_2 +\Psi)\,.
\end{align*}
Hence
\begin{align*}
&e^{-k}(\pt_t v_2+\lambda_2\pt_r v_2)\\
&= \pt_t\pt_r w_2 +\lambda_2\pt^2_r w_2 +\lambda_2\pt_r \Psi +\pt_1\Psi(\pt_t w_1+\lambda_2\pt_rw_1) +\pt_2\Psi (\pt_t w_2+\lambda_2\pt_r w_2)\\
&\quad+(\pt_r w_2+\Psi)(\pt_1 k (\pt_t w_1+\lambda_2\pt_r w_1) +\pt_2 k (\pt_t w_2+\lambda_2\pt_r w_2))\,.
\end{align*}
Note that  $\pt_t w_1 +\lambda_2\pt_r w_1 = f+(\lambda_2-\lambda_1)\pt_r w_1$, and $
\pt_t w_2 +\lambda_2\pt_r w_2 = -f $ so that  
\begin{align*}
&e^{-k}(\pt_t v_2+\lambda_2\pt_r v_2)\\
&= \pt_t\pt_r w_2 +\lambda_2\pt^2_r w_2 +\lambda_2\pt_r \Psi +(f+(\lambda_2-\lambda_1)\pt_r w_1 )\pt_1\Psi -f\pt_2\Psi \\
&\quad+(\pt_r w_2+\Psi)((f+(\lambda_2-\lambda_1)\pt_r w_1)\pt_1 k  -f\pt_2 k )\,.
\end{align*}
Derivating with respect to time the equation in $w_2$, we obtain:
\begin{align*}
&\pt_t\pt_r w_2 +\lambda_2 \pt_r^2 w_2 \\
&= -\pt_r f-\pt_1 f \pt_r w_1 - \pt_2 f\pt_r w_2 -\pt_1\lambda_2 \pt_r w_1 \pt_r w_2 -\pt_2\lambda_2 (\pt_r w_2)^2 \\
&=-\pt_r f - \pt_1 f \pt_r w_1- \pt_2 f(e^{-k }v_2 -\Psi) -\pt_1\lambda_2 (e^{-k} v_2-\Psi)\pt_r w_1 -\pt_2\lambda_2 (e^{-k }v_2 -\Psi)^2 \\
&=- e^{-2k}v_2^2  \pt_2\lambda_2 +e^{-k }v_2(-\pt_2 f+2 \Psi\pt_2\lambda_2) +(-\pt_r f+\Psi \pt_2 f -\Psi^2\pt_2\lambda_2 )\\
&\quad +\pt_r w_1 (-\pt_1 f+\Psi \pt_1\lambda_2) -\pt_1\lambda_2 e^{-k} v_2 \pt_r w_1\,. 
\end{align*}
Replacing, we get
\begin{align*}
&e^{-k}(\pt_t v_2+\lambda_2\pt_r v_2)\\
&=- e^{-2k}v_2^2  \pt_2\lambda_2 +e^{-k }v_2(-\pt_2 f+2 \Psi\pt_2\lambda_2) +(-\pt_r f+\Psi \pt_2 f -\Psi^2\pt_2\lambda_2 )\\
&\quad +\pt_r w_1 (-\pt_1 f+\Psi \pt_1\lambda_2) -\pt_1\lambda_2 e^{-k} v_2 \pt_r w_1  +\lambda_2\pt_r \Psi +(f+(\lambda_2-\lambda_1)\pt_r w_1 )\pt_1\Psi -f\pt_2\Psi \\
&\quad+e^{-k}v_2 ((f+(\lambda_2-\lambda_1)\pt_r w_1)\pt_1 k  -f\pt_2 k )\\
&= - e^{-2k}v_2^2  \pt_2\lambda_2 +e^{-k }v_2(-\pt_2 f+2 \Psi\pt_2\lambda_2 +f (\pt_1 k-\pt_2k)) \\
&\quad +(-\pt_r f+\Psi \pt_2 f -\Psi^2\pt_2\lambda_2+\lambda_2\pt_r \Psi +f(\pt_1\Psi-\pt_2\Psi) )\\
&\quad +\pt_r w_1 (-\pt_1 f+\Psi \pt_1\lambda_2+(\lambda_2-\lambda_1)\pt_1\Psi )   +e^{-k}v_2 \pt_r w_1((\lambda_2-\lambda_1)\pt_1 k -\pt_1\lambda_2 )\,.
\end{align*}
With our choice for $k $ and $\Psi$, the last line above vanishes.
\end{proof}

Similarly as above, we want to derive a diagonal system on $rv_1$ and $rv_2$.
\begin{corollary}
With the notation of Lemma \ref{lem:v1v2}, we have for all $\ell\in \naturali$, 
\begin{equation}\label{eq:rv}
\left\{
\begin{array}{lcl}
\pt_t(r^\ell v_1) +\lambda_1(w) \pt_r (r^\ell v_1) &=& \frac{a_0(r,w)}{r^\ell} (r^\ell v_1)^2 +(a_1(r,w)+\frac{\ell\lambda_1 }{r}) (r^\ell v_1)+r^\ell a_2(r,w)\,,\\
\pt_t (r^\ell v_2) +\lambda_2 (w) \pt_r (r^\ell v_2) &=& \frac{b_0(r,w)}{r^\ell} (r^\ell v_2)^2 +(b_1(r,w) +\frac{\ell\lambda_2 }{r}) (r^\ell v_2)+r^\ell b_2(r,w)\,.\\
\end{array}
\right.
\end{equation}
\end{corollary}

In the following we give some properties of the coefficients $a_0, a_1, a_2, b_0,b_1, b_2$:
\begin{lemma}\label{lem:compab}
Let us consider a Bethe-Weyl gas, satisfying $1< \mathscr{G}< 2$ and such that $(\rho \mapsto \frac{c(\rho)}{\rho})$ is integrable in 0.
There exist smooth functions $\bar a_i : \reali^2\to \reali$ and $\bar b_i : \reali^2\to \reali$ such that, for $i\in \{1,2\}$ the coefficients $a_i$ and $b_i$ defined as in Lemma \ref{lem:v1v2} can be written $a_i(r,w)= \frac{\bar a_i(w)}{r^i}$ and $b_i(r,w)= \frac{\bar b_i(w)}{r^i}$. In particular,   $\bar a_i$ and $\bar b_i$ are not depending on $r$. 

Similarly, there  exist smooth functions $\phi : \reali^2\to \reali$ and $\psi : \reali^2\to \reali$ such that the coefficients $\Phi$ and $\Psi$ can be written $\Phi(r,w)= \frac{\phi(w)}{r}$ and $\Psi(r,w)= \frac{\psi(w)}{r^i}$. In particular $\phi$ and $\psi$ are not depending on $r$. 
\end{lemma}

\begin{proof}
We have $\lambda_1=u-c$, and $u=\frac{w_1+w_2}{2}$, $H=\frac{w_2-w_1}{2}$. Hence $\pt_2 \lambda_1=\frac{1}{2}-\frac{c'}{2H'}$ where $\pt_2\lambda_1$ stands for $\frac{\pt \lambda_1(w_1,w_2)}{\pt w_2}$, and
\[
\pt_2 h =\frac{H''}{4(H')^2}\,.
\]
Finally $h=\frac{1}{2}\ln(H')$ satisfies the equation. In the same way, we obtain $k=h=\frac{1}{2}\ln (H')$.

Using $H'(\rho)=\frac{c(\rho)}{\rho}$ and $\mathscr{G}=\frac{1}{c(\rho)}(c(\rho)\rho)'$, we note that $c'=H' (\mathscr{G}-1)$ and that $H''=\frac{H'}{\rho}(\mathscr{G}-2)$. Then
\begin{align*}
\pt_2(e^h \Phi) &=\frac{\sqrt{H'}}{2c}\frac{d-1}{r}\left(\frac{c}{2}+ u\frac{c'}{2H'}\right)= \frac{d-1}{2r}\left(\frac{1}{2}\sqrt{H'} +u \frac{ c'}{2c\sqrt{H'}} \right)\,.
\end{align*}
Let us  define $g$, $A$ and $B$ such that
\begin{align*}
\frac{\d{}}{\d{\rho}}\left(2\sqrt{H'(\rho)}g(\rho)\right)&= \frac{1}{\rho}\sqrt{H'(\rho)}\,,\\
A&=1+g\,,\\
 \frac{\d{}}{\d{\rho}}\left(2\sqrt{H'(\rho)}B(\rho)\right)&= H'(\rho)\sqrt{H'(\rho)}(1+2g(\rho))\,.
\end{align*}
Then, noting that, for any function $f$ of $u $ and $\rho$ we have 
\[
\pt_2 f= \pt_2 u \,\pt_u f+\pt_2\rho \,\pt_\rho f=\frac{1}{2} \pt_u f+\frac{1}{2 H'(\rho)}\pt_\rho f\,.
\]
Hence, we can check that $\pt_2(e^h \Phi) = \frac{d-1}{2r} \pt_2\left(2\sqrt{H'}u A -2\sqrt{H'} B\right)$ and we can choose
\[
\Phi =\frac{d-1}{r} \left( uA - B\right)\,.
\]
In the same way, we have 
\begin{align*}
\Psi=\frac{d-1}{r}\left( uA +B \right)\,.
\end{align*}

Note that
\begin{align*}
\pt_\rho A&=\frac{1}{2\rho} \left(1-g(\mathscr{G}-2)\right)\,,& 
\pt_\rho B&=\frac{1}{2\rho} \left(c(1+2g) -B(\mathscr{G}-2)\right)\,.
\end{align*}
Let us now compute the expression of $a_0, a_1, a_2$:
\begin{align*}
a_0&=-e^{-h}\pt_1\lambda_1=-\frac{1}{\sqrt{H'}}\left(\frac{1}{2}+\frac{c'}{2H'}\right)=\frac{-1}{\sqrt{H'}}\frac{\mathscr{G}}{2}\leq 0\,,\\
a_1&=\frac{d-1}{2r}(c-u(\mathscr{G}-1)) + \mathscr{G}\frac{d-1}{r} \left( uA - B \right) -\frac{d-1}{2r} u(\mathscr{G}-2)\\
&=\frac{d-1}{r}\left[\frac{c}{2} -B\mathscr{G} +\frac{u}{2}\left(3+2\mathscr{G}g  \right)\right]\,,
\end{align*}
and
\begin{align*}
a_2&=\frac{(d-1)\sqrt{H'}}{r^2}\big[ -uc- \frac{d-1}{2}\left((uA-B)(c-u(\mathscr{G}-1)) +\mathscr{G}(uA-B)^2\right)\\
&\quad -\frac{d-1}{2} u( u (1-(\mathscr{G}-2)(A-1)) +( B(\mathscr{G}-2)-c(2A-1 ) ) ) \\
&\quad -(u-c)(uA-B)\big]\\
&=\frac{(d-1)\sqrt{H'}}{r^2}\big[ -uc-(u^2A -u(B+cA)+cB  )  \\
&\qquad - \frac{d-1}{2}\big(  u(Ac +B (\mathscr{G}-1))   -Bc-u^2 A (\mathscr{G}-1)  +\mathscr{G}(u^2A^2 -2uAB+B^2)\\
&\quad \qquad +  u^2 (1-(\mathscr{G}-2)(A-1)) +u ( B(\mathscr{G}-2)-c(2A-1 ) ) )\big)\big] \\
\end{align*}
Let us now consider the expression of $b_0$, $b_1$, $b_2$ as given in (\ref{eq:B0B1B2}). 
\begin{align*}
b_0&= a_0=\frac{-\mathscr{G}}{2\sqrt{H'}}\,,\\
b_1&= \frac{d-1}{2r}\left( u(2\mathscr{G}(A-1) +3)  +2\mathscr{G}B -c\right)\,,\\
b_2&=\frac{(d-1)\sqrt{H'}}{r^2}\big[
-u^2A- u(c(A-1) +B)-cB + \frac{d-1}{2}  \big( u^2(-\mathscr{G} A^2+ (2\mathscr{G}-3)A +1-\mathscr{G}) \\
&\qquad \qquad \qquad
+u( B(2\mathscr{G}-3-2\mathscr{G}A) -c(A-1) ) +Bc -\mathscr{G}B^2\big)
\big]\,.
\end{align*}

\end{proof}

\subsection{$\C1$ estimate on $w_1$ and $w_2$}

We derive now $\C1$-estimates  for $w_1$ and $w_2$. First we derive an upper bound on $\pt_r w_1$ and $\pt_r w_2$.
\begin{lemma}\label{lem:up}
Let us consider a Bethe-Weyl fluid and let $w=(w_1, w_2)$ be a regular solution of (\ref{eq:red_source}) with a time of existence $\mathcal{T}>0$. Let $\alpha, R>0$ with $\alpha<\mathcal{T}$.  
Let $t\mapsto X_1(t)$ and $t\mapsto X_2(t)$ be the characteristics defined as in (\ref{def:car}) and passing through $(\alpha, R)$. Let $T$ be such that $T>\alpha$ and $T<\mathcal{T}$. We assume that $X_1(t), X_2(t)$ are well-defined on $[\alpha,T]$. Let furthermore $v_1$, $v_2$ be defined as in (\ref{def:v}). Then, for all $t\in [0,\mathcal{T}[$, we have
\begin{eqnarray*}
v_1(t, X_1(t))e^{-\int_\alpha^t a_1(s, X_1(s))\d{s}}&\leq &v_1(\alpha, X_1(\alpha)) +\int_\alpha^t a_2(s, X_1(s))e^{\int_s^t a_1(\tau, X_1(\tau))\d{\tau}}\,,\\
v_2(t, X_2(t))e^{-\int_\alpha^t b_1(s, X_2(s))\d{s}}&\leq &v_2(\alpha, X_2(\alpha)) +\int_\alpha^t b_2(s, X_2(s))e^{\int_s^t b_1(\tau, X_2(\tau))\d{\tau}}\,.
\end{eqnarray*}

More generally, for any $\ell\in \naturali$,  
\begin{eqnarray*}
&&X_1(t)^\ell v_1(t, X_1(t))\\
&\leq & X_1(\alpha)^\ell v_1(\alpha, X_1(\alpha))e^{\int_\alpha^t (\frac{\ell\lambda_1}{X_1(s)} +a_1(s, X_1(s)))\d{s}} + \int_\alpha^t X_1(s)^\ell  a_2(s, X_1(s))e^{\int_s^t  (\frac{\ell \lambda_1}{X_1(\tau)} +a_1(\tau, X_1(\tau)))\d{\tau}}\,,\\
&&X_2(t)^\ell v_2(t, X_2(t))\\
&\leq &X_2(\alpha)^\ell v_2(\alpha, X_2(\alpha)) e^{\int_\alpha^t  (\frac{\ell \lambda_2}{X_2(s)} +b_1(s, X_2(s)))\d{s}} + \int_\alpha^t X_2^\ell  (s) b_2(s, X_2(s))e^{\int_s^t (\frac{\ell \lambda_2}{X_2(\tau)} +b_1(\tau, X_2(\tau)))\d{\tau}}\,.
\end{eqnarray*}
\end{lemma}

\begin{proof}
Let us consider $v_1$, the argument being similar for $v_2$. Let us denote $y_1(t)=v_1(t, X_1(t))$.
The coefficient $a_0$ being non-positive, we have by Lemma \ref{lem:v1v2} $y_1'(t)\leq a_1 y_1 +a_2$. Hence, 
\[
(y_1 e^{-\int_\alpha^t a_1(s, X_1(s))\d{s}})'\leq  a_2(t,X_1(t)) e^{-\int_\alpha^t a_1(s, X_1(s))\d{s}}\,.
\]
Integrating we obtain the desired estimate. The second set of estimates is obtained in the same way considering the system (\ref{eq:rv}) instead of (\ref{eq:v1v2}).
\end{proof}

We now derive a lower bound on $v_1$ and $v_2$. First we consider the case in which the initial condition $v_1(\alpha, R)$ is positive.
\begin{lemma}\label{lem:w'pos} 
Let us consider a Bethe-Weyl fluid and let $w=(w_1, w_2)$ be a regular solution of (\ref{eq:red_source}) with a time of existence $\mathcal{T}>0$. Let $\alpha, R>0$ with $\alpha<\mathcal{T}$.  
Let $t\mapsto X_1(t)$ and $t\mapsto X_2(t)$ be the characteristics defined as in (\ref{def:car}) and passing through $(\alpha, R)$. Let $T$ be such that $T>\alpha$ and $T<\mathcal{T}$.  We assume that $X_1(t), X_2(t)$ are well-defined on $[\alpha,T]$. 

Let us denote 
$\bar a_1= r a_1$, $\bar a_2=r^2a_2$, $\bar b_1=rb_1$, $\bar b_2=r^2b_2$ as in Lemma \ref{lem:compab}, and 
\begin{align}
A_0&=\max_{[0,T]}\modulo{a_0(t,X_1(t))}\,, &A_1&=\max_{[0,T]}\modulo{\bar a_1(t,X_1(t))}\,, &A_2&=\max_{[0,T]}\modulo{\bar a_2(t,X_1(t))}\,, \label{eq:A0A1A2}\\
B_0&=\max_{[0,T]}\modulo{b_0(t,X_2(t))}\,, &B_1&=\max_{[0,T]}\modulo{\bar b_1(t,X_2(t))}\,, &B_2&=\max_{[0,T]} \modulo{\bar b_2(t, X_2(t))}\,.\label{eq:B0B1B2}
\end{align} 
We also denote, for all $\zeta_0, \zeta_1, \zeta_2$ positive, $x(\zeta_0, \zeta_1, \zeta_2)$ the unique positive solution of 
\begin{align}
Q(r)=\frac{\zeta_1}{\sqrt{\zeta_0 \zeta_2}}\,, \qquad \textrm{ where } Q(r):=re^r\,. \label{eq:phi}
\end{align} 
Furthermore, we denote 
\begin{align}
K_a(\theta)= \left(\int_\alpha^\theta\modulo{a_2(t, X_1(t))}\d{t} \right)\cdot \exp{ (\int_\alpha^\theta \modulo{a_1(t, X_1(t))}\d{t})}\,,\label{eq:Ka}\\
K_b(\theta)= \left(\int_\alpha^\theta\modulo{b_2(t, X_2(t))}\d{t}  \right)\cdot \exp{ (\int_\alpha^\theta \modulo{b_1(t, X_2(t))}\d{t})}\,.\label{eq:Kb}
\end{align}
Then:
\begin{enumerate}
\item If  $v_1(\alpha, R)\geq 0$, then for all $t\in \left[\alpha, \alpha+\frac{x(A_0, A_1, A_2)}{A_1}\min_{[0,T]} X_1 (t)\right]$, we have
\begin{align*}
\frac{-K_a(t)}{1-K_a(t)\int_\alpha^t \modulo{a_0(\tau, X_1(\tau))}e^{\int_\alpha^\tau \modulo{a_1(u,X_1(u))} \d{u}}\d{\tau} }\leq &v_1(t, X_1(t))e^{-\int_\alpha^t a_1(s, X_1(s))\d{s}} \\
\leq &v_1(\alpha, X_1(\alpha ))+ K_a(t) \,.
\end{align*}
\item  If $v_2(\alpha, R)\geq 0$, then for all $t\in \left[\alpha, \alpha+\frac{x(B_0, B_1, B_2)}{B_1}\min_{[0,T]} X_2 (t)\right]$, we have
\begin{align*}
\frac{-K_b(t)}{1-K_b(t)\int_\alpha^t \modulo{b_0(\tau, X_2(\tau))}e^{\int_\alpha^\tau \modulo{b_1(u,X_2(u))} \d{u}}\d{\tau} }
\leq & v_2(t, X_2(t)) e^{-\int_\alpha^t b_1(s, X_2(s))\d{s}}\\ 
\leq &v_2(\alpha, X_2(\alpha ))+ K_b(t) \,.
\end{align*}
\end{enumerate}
\end{lemma}

\begin{remark}
The quantities $A_i$ and $B_i$, for $i\in \{0,1,2\}$,  are not depending directly on $r$ but only on $\norma{w_1}_{\L\infty}$ and $\norma{w_2}_{\L\infty}$.  In the case $\frac{\d{X_2}}{\d{t}}>0$ and $\alpha=0$, we obtain a lower bound for the time of existence of $\pt_r w_2$ 
\[
T_{ex}(\pt_r w_2) \geq \frac{x(B_0, B_1, B_2 )}{ B_1} R\,.
\]
\end{remark}

\begin{remark}\label{rem:rv1}
The same lemma can be applied to (\ref{eq:rv}) with 
\begin{align*}
\tilde a_0&=\frac{a_0}{r^\ell}\,, &\tilde a_1&=a_1+\frac{\ell \lambda_1}{r}\,, &\tilde a_2&=r^\ell a_2\,,\\
\tilde b_0&=\frac{b_0}{r^\ell}\,, &\tilde b_1&=b_1+\frac{\ell \lambda_2}{r}\,, &\tilde b_2&=r^\ell b_2\,.
\end{align*}
 The time of validity of the estimates remains similar, replacing  $A_1$ by  $\tilde A_1=\max_{[0,T]}(\bar a_1+\lambda_1)(t,X_1(t)))$ and  $B_1$ by $\tilde B_1=\max_{[0,T]}(\bar b_1+\lambda_2)(t,X_2(t)))$, with $A_0$, $B_0$, $A_2$, $B_2$ unchanged.
\end{remark}

\begin{proof}
First note that the upper bounds come directly from Lemma \ref{lem:up}.

Let us consider the equation on $v_2$. Without a lot of changes we can adapt the following to $v_1$.  Let us denote $y_2(t)=v_2(t, X_2(t))$ where  $v_2$ is defined in  in (\ref{def:v}). 
The equation on $v_2$ in (\ref{eq:v1v2}) writes, with  $y_2(t)=v_2(t,X_2(t))$,
\begin{align}
y_2'=b_0 y_2^2+b_1 y_2+b_2\,,\label{edo:y2}
\end{align}
where $b_0\leq 0$.

According to a Lemma  \ref{lem:hor2} (see also H\"ormander \cite{hormander}), if we have $y_2(\alpha)\geq 0$ and, for $\theta>\alpha$
\begin{align}
 \int_\alpha^\theta\modulo{b_2(t, X_2(t))}\d{t} \int_\alpha^\theta \modulo{b_0(t, X_2(t))}\d{t}\, \exp{ (2\int_\alpha^\theta \modulo{b_1(t, X_2(t))}\d{t})}&<1\,,\label{cond:hor}
\end{align}
then the equation (\ref{edo:y2})  with initial condition $ y_2(\alpha)= v_2(\alpha, X_2(\alpha))$ in $t=\alpha$ admits a solution on  $[\alpha, \theta]$ and, with $K_b$ as in (\ref{eq:Kb}), 
we have the estimate 
\begin{align*}
&\frac{- K_b(\theta )}{1- K_b (\theta)\int_\alpha^\theta \modulo{b_0(t, X_2(t))}\d{t} \exp{ (\int_\alpha^\theta \modulo{b_1(t, X_2(t))}}\d{t})}\\
\leq &\; y_2(\theta) e^{-\int_\alpha^\theta b_1(s, X_1(s))\d{s}} \leq v_2(\alpha, X_2(\alpha ))+ K_b(\theta) \,.
\end{align*}
Note, thanks to Lemma \ref{lem:compab}, that we have
\begin{align*}
 b_0(r,w)&=\bar b_0(w)\,,& b_1(r,w)&= \frac{\bar b_1(w)}{r}\,,&b_2(r,w)&=\frac{\bar b_2(w)}{r^2}\,.
\end{align*}
As we are not considering the 1-D case,  $b_0, b_2$ are not constantly zero, see the expression of $b_0, b_2$ in  Lemma \ref{lem:v1v2}. 

We have:
\begin{align*}
&\int_\alpha^\theta\modulo{b_2(t,X_2(t))}\d{t} \int_\alpha^\theta \modulo{b_0(t, X_2(t))}\d{t} \exp{ (2\int_\alpha^\theta \modulo{b_1(t, X_2(t))}\d{t})}\\
\leq& 
\frac{B_0 B_2}{ B_1^2} \, \left[\frac{B_1}{\min_{[\alpha,\theta]} (X_2(t))}(\theta-\alpha)\exp\big[{\frac{B_1}{\min_{[\alpha,\theta]} (X_2(t))}(\theta-\alpha)}\big]\right]^2\\
= &  \frac{B_0 B_2}{ B_1^2}  Q\left(\frac{B_1}{\min_{[\alpha,\theta]} (X_2(t))}(\theta-\alpha)\right)^2\,.
\end{align*}
Hence, the condition (\ref{cond:hor}) is satisfied for all $\theta\leq T$ such that $ Q\left(\frac{B_1}{\min_{[\alpha,\theta]} (X_2(t))}(\theta-\alpha)\right)< \frac{ B_1}{\sqrt{B_0 B_2}}  $. So, it is sufficient to ask that $\theta$ satisfies:
\[
\theta\leq  \alpha+\min_{t\in [\alpha,\theta]} (X_2(t))\,\frac{x(B_0,B_1,B_2)}{B_1}\,.
\]
\end{proof}

Let us now derive a lower bound on $v_1$ and $v_2$ in the case the initial condition is negative.
\begin{lemma}\label{lem:w'neg} 
Let $\alpha, R>0$. 
With the same notations as introduced in Lemma \ref{lem:w'pos}, we obtain:
\begin{enumerate}
\item If $v_1(\alpha, R)\leq 0$, then $v_1$ is well-defined on every interval $[\alpha, \theta]$ such that
\[
\theta\leq \alpha +\frac{\min_{t\in [\alpha, \theta]} X_1(t)}{A_1}Q^{-1}\left(\Theta_+\right)\,,
\]
where 
\begin{align}
\Theta_+ = \frac{A_1}{\sqrt{A_0 A_2}} \,\frac{1}{\left(1+\frac {\sqrt{A_0}} {\sqrt{A_2} } ( \modulo{v_1(\alpha, R)} \min_{t\in [\alpha, \theta]} X_1(t) )\right)} \,, \label{eq:Theta}
\end{align} 
and we have the estimate
\begin{eqnarray*}
\frac{-(\modulo{v_1(\alpha , R)} +K_a)}{1-(\modulo{v_1(\alpha, R)}+ K_a)\int_\alpha^\theta \modulo{a_0(t, X_1(t))}e^{\int_\alpha^t \modulo{a_1(u,X_1(u))} \d{u}}\d{t} }\\
\leq v_1(t, X_1(t))e^{\int_\alpha^t a_1(s, X_1(s))\d{s}}\leq K_a + v_1(\alpha, R)\,. 
\end{eqnarray*}
\item If $v_2(\alpha, R)\leq 0$, then $v_2$ is well-defined on every interval $[\alpha, \theta]$ such that
\[
\theta\leq \alpha +\frac{\min_{t\in [\alpha, \theta]} X_2(t)}{B_1}Q^{-1}\left(\Xi_+\right)\,,
\]
where
\begin{align}
\Xi_+ = \frac{B_1}{\sqrt{B_0 B_2}} \,\frac{1}{\left(1+\frac {\sqrt{B_0}} {\sqrt{B_2} } ( \modulo{v_2(\alpha)} \min_{t\in [\alpha, \theta]} X_2(t) )\right)} \,, \label{eq:Chi}
\end{align} 
and we have the estimate
\begin{eqnarray*}
&&\frac{-(\modulo{v_2(\alpha, R) }+K_b)}{1-(\modulo{v_2(\alpha, R)}+ K_b)\int_\alpha^\theta \modulo{b_0(t, X_2(t))}e^{\int_\alpha^t \modulo{b_1(u,X_2(u))} \d{u}}\d{t} }\\
&& \leq v_2(t, X_2(t)) e^{\int_\alpha^t b_1(s, X_2(s))\d{s}}
 \leq K_b + v_2(\alpha, R)\,. 
\end{eqnarray*}
\end{enumerate}
\end{lemma}

\begin{proof}
First note that the upper bounds come directly from Lemma \ref{lem:up}.

Let us consider $v_1$. The same computations apply to $v_2$ after small changes. Let us denote $z_1(t)=-v_1(t,X_1(t))$. Then, by hypothesis $z_1(\alpha)\geq 0$. According to Lemma \ref{lem:v1v2}, $z_1$ satisfies the ODE:
\[
z_1'=-a_0 z_1^2+a_1 z_1 -a_2\,,
\]
where $-a_0\geq 0$ (see Lemma \ref{lem:compab}).
Let us introduce $K_a$ as in (\ref{eq:Ka}).
To apply Hörmander Lemma (see Lemma \ref{lem:hor2}), conditions (\ref{hyp:1}) and (\ref{hyp:2}) have to be satisfied. 
Note that, since $-a_0\geq 0$, condition (\ref{hyp:1}) implies condition (\ref{hyp:2}) so it is sufficient to see what is a sufficient condition allowing (\ref{hyp:1}) to be satisfied. Condition (\ref{hyp:1}) is equivalent to 
\[
\left(z_1(\alpha) +K_a\right) \int_\alpha^\theta \modulo{a_0(t, X_1(t))} \d{t} e^{\int_\alpha^\theta \modulo{a_1(t, X_1(t))} \d{t}} -1 <0\,.
\]
It is sufficient to have
\begin{align*}
z_1(\alpha) A_0(\theta-\alpha) e^{\frac{A_1}{\min_t X_1(t)}(\theta-\alpha)}+\frac{A_0 A_2}{\min_t X_1(t)^2}(\theta-\alpha)^2 e^{\frac{2A_1}{\min_t X_1(t)}(\theta-\alpha)} -1<0\,.
\end{align*}
Let us denote $\Theta=\frac{A_1}{\min_{t\in [\alpha, \theta]} X_1(t)}(\theta-\alpha)e^{\frac{A_1}{\min_{t\in [\alpha, \theta]} X_1(t)}(\theta-\alpha)}=Q(\frac{A_1}{\min_{t\in [\alpha, \theta]} X_1(t)}(\theta-\alpha))$, where $Q$ is defined as in (\ref{eq:phi}). Then to satisfy condition (\ref{hyp:1}) it is sufficient to have
\begin{align*}
z_1(\alpha)\frac{A_0}{A_1} \min_{t\in [\alpha, \theta]} \{ X_1(t)\} \Theta+\frac{A_0 A_2}{A_1^2} \Theta^2-1<0\,.
\end{align*}
Let $\Delta=\left(z_1(\alpha)\frac{A_0}{A_1} \min_t \{ X_1(t)\}\right)^2+ 4 \frac{A_0 A_2}{A_1^2}$ be the discriminant of this equation in $\Theta$. This equation admits two distinct roots, one positive and one negative.  Let us denote $\xi_+$ the positive root. Then we have
\begin{align*}
\xi_+ &= \frac{A_1 z_1(\alpha) \min_{t\in [\alpha, \theta]} X_1(t) }{2A_2}\left(-1+\sqrt{1+\frac{4 A_2}{ z_1(\alpha)^2 A_0 \min_{t\in [\alpha, \theta]} X_1(t)^2}}\right)\\
&= \frac{A_1}{\sqrt{A_0 A_2}} \frac{1} { \sqrt{1+\frac{A_0}{4A_2} ( z_1(\alpha) \min_{t\in [\alpha, \theta]} X_1(t) )^2} +\frac {\sqrt{A_0}} {2\sqrt{A_2} } ( z_1(\alpha) \min_{t\in [\alpha, \theta]} X_1(t) ) } \\
&\geq   \frac{A_1}{\sqrt{A_0 A_2}} \frac{1}{1+\frac {\sqrt{A_0}} {\sqrt{A_2} } ( z_1(\alpha) \min_{t\in [\alpha, \theta]} X_1(t) )} =: \Theta_+\,.
\end{align*} 
Then (\ref{hyp:1}) is satisfied if $0\leq \Theta \leq \Theta_+$.
Hence, the application $Q$ defined in (\ref{eq:phi}) being strictly increasing on $\reali_+$, it is sufficient to have
\[
\theta\leq \alpha +\frac{\min_{t\in [\alpha, \theta]} X_1(t)}{A_1}Q^{-1}\left(\Theta_+\right)\,.
\]
\end{proof}

\section{Construction of a shock wave}\label{sec:constr}
\subsection{Rankine-Hugoniot conditions}\label{sectionRH}
The Rankine-Hugoniot conditions (cf. \cite[p. 312]{Benz}) appear when we consider weak and piecewise smooth solution for first order systems. For the Euler equations (\ref{eq:isentropique}), these conditions are written:
\[ \left\{
\begin{array}{l}
-{U}[\rho]+[\rho ({u}\cdot\nu)]=0\,,\\
-{U}[\rho {u}] + [\rho ({u}\cdot\nu)u+ p\nu]=0\,,
\end{array} \right.
\]
through a discontinuity with normal vector $\nu$  and with normal speed $U$.
The usual notation $[\cdot]$ stands for the jump between the two limit values at the both sides of the discontinuity. We denote $u^+$ the limit of $u$ from the right and $u^-$ the limit of $u$ from the left.
In the spherical case, these conditions become
\begin{equation}\label{RH}
 \left\{\begin{array}{l}
-{U}[\rho]+[\rho {u}]=0\,,\\
-{U}[\rho {u}] + [\rho u^2+ p]=0\,.
 \end{array}\right.
\end{equation}
A weak solution with a discontinuity is called \emph{shock} when $U$ differs from the speed of the fluid on the both sides of the discontinuity. Note that the first condition of (\ref{RH}) gives us $\rho^+(U-u^+) =\rho^-(U-u^-)$ so that $U-u^+$ and $U-u^-$ have the same sign. If we assume $U-u^\pm\geq 0$ then the shock moves from the left to the right.

Let $W^\pm = U - u^\pm$. With some classical computations, we get that (\ref{RH}) is equivalent to 
\begin{align}
j:=\rho^+(U-u^+) =\rho^-(U-u^-)\,,\qquad [p+\rho W^2]=0\,.\label{eq:RH}
\end{align}
Hence, we get 
\begin{align*}
u^+-u^-=j\left(\frac{1}{\rho^-}-\frac{1}{\rho^+}\right)\,,\qquad p^+-p^-=j^2\left(\frac{1}{\rho^-}-\frac{1}{\rho^+}\right)\,.
\end{align*}
Finally, we obtain $(u^+-u^-)^2=(p^+-p^-)\left(\frac{1}{\rho^-}-\frac{1}{\rho^+}\right)\geq 0$.
Let us recall furthermore the Lax entropy conditions for a 2-shock (see \cite{Sr1}):
\begin{align}
\rho^- > \rho^+\,,& \qquad\qquad\qquad p^- >p^+\,,& u^-&> u^+\,,\label{rholeqrho+}\\
\lambda_1(w^-)&<U(w^+,w^-)<\lambda_2( w^-)\,,& U(w^+)&\geq \lambda_2(w^+)\,.\label{eq:U}
\end{align}

Finally, the jump conditions at the shock are 
\begin{align}
 U&=\frac{\rho^+u^+-\rho^-u^-}{\rho^+-\rho^-}\,,&u^--u^+&=\sqrt{(p^+-p^-)\left(\frac{1}{\rho^-}-\frac{1}{\rho^+}\right)}\,,& \rho^- &>\rho^+\,.\label{eq:cond_saut}
\end{align}

%
\begin{proposition}
For a Bethe-Weyl gas, the  Rankine-Hugoniot  and Lax shock conditions can be reduced to (\ref{eq:cond_saut}).
\end{proposition}

\begin{proof}
By (\ref{eq:cond_saut}), we have immediatly $u^-> u^+$ and $p^->p^+$ since $\frac{\pt p}{\pt \rho}=c^2\geq 0$. Hence (\ref{rholeqrho+}) is satisfied.

Let us prove (\ref{eq:U}). First, note that
\[
U-u^-=\frac{\rho^+ \sqrt{F(\rho^-, \rho^+)}}{\rho^--\rho^+}\,,
\]
where $F$ is defined as in (\ref{eq:F}). Hence $U-u^-\geq 0\geq - c^-$ and $U\geq u^--c^-=\lambda_1^-$. 

Let us compute now $U-u^--c^-$:
\begin{align*}
U-u^--c^- &= \frac{\rho^+}{\rho^--\rho^+}\sqrt{\rho^-c^-\left(\frac{1}{\rho^+}-\frac{1}{\rho^-}\right)}
\left(\sqrt{\frac{1}{\rho^-c^-}\left(p^- -p^+\right)}-\sqrt{\rho^-c^-\left(\frac{1}{\rho^+}-\frac{1}{\rho^-}\right)}\right)\\
&=\frac{\rho^+}{\rho^--\rho^+}\sqrt{\rho^-c^-\left(\frac{1}{\rho^+}-\frac{1}{\rho^-}\right)}\frac{
 \frac{1}{\rho^-c^-}\left( p^- -p^+ -(\rho^-c^-)^2 \left(\frac{1}{\rho^+}-\frac{1}{\rho^-}\right)\right)
}{\left(\sqrt{\frac{1}{\rho^-c^-}\left(p^- -p^+\right)}+\sqrt{\rho^-c^-\left(\frac{1}{\rho^+}-\frac{1}{\rho^-}\right)}\right)}
\end{align*}
Let us denote $\xi(\rho^-, \rho^+)=p^- -p^+ -(\rho^-c^-)^2 \left(\frac{1}{\rho^+}-\frac{1}{\rho^-}\right)$. We have, for $\rho\geq \rho^+$
\[
\frac{\pt \xi(\rho, \rho^+)}{\pt{\rho}} =-2\rho c^2\mathscr{G}\left(\frac{1}{\rho^+}- \frac{1}{\rho}\right)\leq 0\,.
\]
Integrating on $[\rho^+, \rho^-]$, we obtain $\xi(\rho^-, \rho^+)\leq \xi(\rho^+, \rho^+)=0$.
Finally, we have $U-u^--c^-\leq 0$. 

A similar computation prove that $U-u^+-c^+\geq 0$.
\end{proof}

Let us invert the second of the jump conditions  (\ref{eq:cond_saut}), in order to express it as a condition on $w_1^-$, depending on $w_2^-, w_1^+, w_2^+$.
\begin{lemma}\label{lem:g} 
We assume that there exists $(w_{1,0}^-, w_{2,0}^-)$ and $(w_{1,0}^+, w_{2,0}^+)$ such that $\mathcal{F}(w_0^-, w_0^+)=0$. Then, there exists $g \in \C1(\reali^3, \reali)$ such that the second condition of (\ref{eq:cond_saut}) is equivalent to 
the compatibility condition 
\begin{equation}
w_1^-=g(w_2^- , w_1^+, w_2^+) \,,\label{eq:cond:bord}
\end{equation}
as long as the condition $\rho^->\rho^+$ is satisfied. That is to say, as long as $g(w_2^-, w^+)< w_2^--w_2^++w_1^+$.
\end{lemma}

\begin{proof}
We can write (\ref{eq:cond_saut}) as the following:
\begin{equation}\label{def:F}
 \mathcal{F}(w^-,w^+):= u^- -u^+ -\sqrt{F(\rho^-,\rho^+)}=0 \,,
\end{equation}
with $F(\rho,\rho^+)=(\frac{1}{\rho}-\frac{1}{ \rho^+})( p(\rho^+) -p(\rho))$ as in (\ref{eq:F}).  Since by hypothesis $\rho^+<\rho^-$, we have
\begin{align}
\frac{\pt \mathcal{F}}{\pt w_1^-}
&= \frac{1}{4\sqrt{F}}\left(\sqrt{\rho^- c^- (\frac{1}{\rho^+}-\frac{1}{\rho^-} ) } +\sqrt{\frac{1}{\rho^- c^-} (p^- -p^+)} \right)^2 \nonumber\\
&> 0\,. \label{Fpos}
\end{align}
According to the hypotheses $\mathcal{F}(w^-(0,R_0),w^+(0,R_0))=0$, that is to say that the jump condition  given by the second equation of   (\ref{eq:cond_saut}) is satisfied at time $t=0$. By the implicit function Theorem, there exist (locally) a unique function $g(w_2^-, w^+)$ such that 
$$\mathcal{F}( g(w_2^-, w^+),w_2^-, w^+)=0\,.$$

Furthermore
\begin{align*}
\frac{\pt \mathcal{F}}{\pt w_2^-}
&= \frac{-1}{4\sqrt{F}}\left(\sqrt{\rho^- c^- (\frac{1}{\rho^+}-\frac{1}{\rho^-} ) } -\sqrt{\frac{1}{\rho^- c^-} (p^- -p^+)} \right)^2 \leq 0\,, \\
\frac{\pt \mathcal{F}}{\pt w_1^+}
&= \frac{-1}{4\sqrt{F}}\left(\sqrt{\rho^+ c^+ (\frac{1}{\rho^+}-\frac{1}{\rho^-} ) } +\sqrt{\frac{1}{\rho^+ c^+} (p^- -p^+)} \right)^2\leq 0\,, \\
\frac{\pt \mathcal{F}}{\pt w_2^+}
&= \frac{1}{4\sqrt{F}}\left(\sqrt{\rho^+ c^+ (\frac{1}{\rho^+}-\frac{1}{\rho^-} ) } -\sqrt{\frac{1}{\rho^+ c^+} (p^- -p^+)} \right)^2 \geq 0\,, 
\end{align*}
and we can see that the sign of the partial derivatives remain constant. 
As $(\frac{\pt g}{\pt w_2^-},\frac{\pt g}{\pt w_1^+}, \frac{\pt g}{\pt w_2^+} )=\frac{-1}{\pt_{w_1^-}\mathcal{F}}( \pt_{w_2^-}\mathcal{F}, \pt_{1^+}\mathcal{F}, \pt_{2^+} \mathcal{F})$ is well defined as long as $\rho^->\rho^+$.
Hence, the implicit function is defined as long as the condition $\rho^->\rho^+$ is satisfied.
\end{proof}

\begin{remark}\label{rem:condDens}
Note that the condition $\rho^->\rho^+$ along the discontinuity line $\mathcal{K}$ is equivalent to $w_2^->w_2^+$ on $\mathcal{K}$. Indeed, as $H$ is a strictly increasing function of $\rho$, $\rho^->\rho^+$ is equivalent to $H^->H^+$, that is to say $w_2^--w_1^->w_2^+-w_1^+$. On the discontinuity line $\mathcal{K}$, $w_1^-=g(w_2^-,w^+)$. Hence, $\rho^->\rho^+$ is equivalent to $g(w_2^-,w_1^+,w_2^+)<w_2^--w_2^++w_1^+$. 

Let us denote $h(w_2^-,w_1^+, w_2^+)=w_2^--w_2^++w_1^+$. As $\pt_{w_1^-}\mathcal F>0$, then $g(w_2^-, w^+)<h(w_2^-, w^+)$ is equivalent to $\mathcal{F}(g(w_2^-, w^+), w_2^-, w^+)<\mathcal{F}(h(w_2^-, w^+), w_2^-, w^+)$. 
We note that 
\[
u(h, w_2^-)=\frac{1}{2}(h+w_2^-)=\frac{1}{2}(2w_2^- -(w_2^+-w_1^+))=w_2^- -H(w_1^+, w_2^+)
\] 
and 
\[
\rho(h,w_2^-)=H^{-1}\left(\frac{1}{2}(w_2^--h)\right)=H^{-1}\left(\frac{1}{2}(w_2^+-w_1^+)\right)=\rho(w_1^+, w_2^+)\,.
\] 
Hence we obtain that $\mathcal{F}(g, w_2^-, w^+)<\mathcal{F}(h,w_2^-,w^+)$ is equivalent to 
\[
0< (w_2^--H^+)-u^+ =w_2^--(u^++H^+)=w_2^--w_2^+\,.
\]
\end{remark}

\subsection{Angular domain}
We want now to construct a solution by solving three problems : two classical problems with initial conditions obtained by prolongating the initial conditions on the right and on the left of $R_0$; and an angular problem with boundary conditions given in (\ref{eq:cond:bord})--(\ref{eq:x2}). 
According to  T. T.  Li \& W. C. Yu \cite[Chap. 3]{liYu}, this last problem admits a local in time solution. In order to obtain an estimate  on the time of existence we make a priori estimates on the solution.

Let us denote $D_-$ the domain in the $(r,t)$-plan which is bounded on the right by the curve $\mathcal{C}_1$ defined in (\ref{eq:x2}) and by the lines of equations $t=0$, $t=T_*$ (see Figure \ref{fig:angulaire}). In the same way, we denote $D_+$ the domain in the $(r,t)$-plan which is bounded on the left by the curve $\mathcal{K}$ defined in (\ref{eq:x2}) and by the lines of equations $t=0$, $t=T_*$. The domain $D_0$ is the domain in the $(r,t)$-plan which is bounded on the left by the curve $\mathcal{C}_1$ defined in (\ref{eq:x2}) and on the right by the curve $\mathcal{K}$ defined in (\ref{eq:x2}).

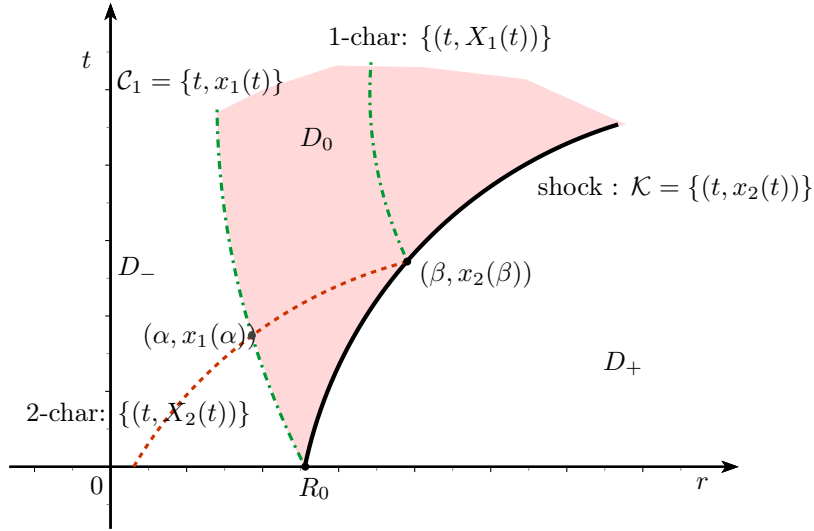
\begin{figure}[ht] 
\begin{center} 
\newrgbcolor{ccttqq}{0.8 0.2 0}
\newrgbcolor{qqzztt}{0 0.6 0.2}
\newrgbcolor{ffwwww}{1 0.4 0.4}
\psset{xunit=1.0cm,yunit=1.0cm,algebraic=true,dotstyle=*,dotsize=3pt 0,linewidth=0.8pt,arrowsize=3pt 2,arrowinset=0.25}
\begin{pspicture*}(-1.33,-0.82)(9.3,6.15)
\psaxes[labelFontSize=\scriptstyle,xAxis=true,yAxis=true,labels=none,Dx=1,Dy=1,ticksize=-2pt 0,subticks=2]{->}(0,0)(-1.33,-0.82)(8.26,6.15)
\pspolygon[linestyle=none,fillstyle=solid,fillcolor=ffwwww,opacity=0.25](2.23,0.71)(1.8,1.94)(1.54,3.12)(1.41,4.38)(1.4,4.7)(2,5)(2.96,5.32)(4.1,5.3)(5.46,5.14)(6.78,4.54)(6.44,4.46)(5.8,4.19)(5.32,3.92)(4.68,3.48)(4.33,3.17)(3.78,2.58)(3.39,2.03)(2.98,1.28)(2.71,0.6)(2.54,0)
\parametricplot[linewidth=1.6pt]{1.8712625613454594}{2.938048166068375}{1*6.03*cos(t)+0*6.03*sin(t)+8.46|0*6.03*cos(t)+1*6.03*sin(t)+-1.22}
\parametricplot[linewidth=1.2pt,linestyle=dashed,dash=2pt 2pt,linecolor=ccttqq]{1.8257925498382732}{2.6098944901921666}{1*5.9*cos(t)+0*5.9*sin(t)+5.39|0*5.9*cos(t)+1*5.9*sin(t)+-2.99}
\parametricplot[linewidth=1.2pt,linestyle=dashed,dash=1pt 2pt 3pt 2pt ,linecolor=qqzztt]{3.151262932604429}{3.6119394248259447}{1*10.69*cos(t)+0*10.69*sin(t)+12.09|0*10.69*cos(t)+1*10.69*sin(t)+4.84}
\parametricplot[linewidth=1.2pt,linestyle=dashed,dash=1pt 2pt 3pt 2pt ,linecolor=qqzztt]{3.0626184080623915}{3.580273898702151}{1*5.24*cos(t)+0*5.24*sin(t)+8.65|0*5.24*cos(t)+1*5.24*sin(t)+4.95}
 \rput[tl](5.62,3.87){{\small shock : $\mathcal{K}=\{(t,x_2(t))\}$}}
\rput[tl](7.7,-0.14){{\small $r$}}
\rput[tl](-0.38,5.51){{\small $t$}}
\rput[tl](-0.27,-0.14){{\small 0}}
\rput[tl](0.09,5.27){{\small $\mathcal{C}_1=\{t,x_1(t)\}$}}
\rput[tl](2.49,4.49){{\small $D_0$}}
\rput[tl](0.09,2.78){{\small $D_-$}}
\rput[tl](6.48,1.53){{\small $D_+$}}
\rput[tl](2.86,5.86){{\small 1-char: $\{(t,X_1(t))\}$}}
\rput[tl](-1.11,0.87){{\small 2-char: $\{(t,X_2(t))\}$}}
\rput[tl](4.06,2.73){{\small $(\beta, x_2(\beta))$}}
\rput[tl](0.42,1.9){{\small $(\alpha,x_1(\alpha))$}}
\rput[tl](2.47,-0.14){{\small $R_0$}}
\psdots(2.56,0)
\psdots(3.9,2.72)
\psdots[linecolor=darkgray](1.86,1.74)
\end{pspicture*}
\caption{Angular Domain and some related curves.}\label{fig:angulaire}
\end{center}
\end{figure}

More precisely,  let $R_0>0$. Let us define the free boundary domain $D_0$ (see Figure \ref{fig:angulaire}) where the boundaries $\mathcal{C}_1=\{(t,x_1(t))\,,\; t\geq 0\}$ and $\mathcal{K}=\{(t,x_2(t))\,,\; t\geq 0\}$ are respectively the curves defined by 
\begin{align}
\frac{\d{x_1}}{\d{t}}&=\lambda_1(w(t,x_1(t))),&x_1(0)&=R_0\,,&
\frac{\d{x_2}}{\d{t}}&=U(t,w(t,x_2(t))),& x_2(0)&=R_0\,.\label{eq:x2}
\end{align}
where $U$ is the speed of a 2-shock linking $w^-(t,x_2(t))$ to $w^+(t,x_2(t))$ and is defined by the first equation of (\ref{eq:cond_saut}).

So that the angular problem is well-posed, we need to add some boundary conditions on $\mathcal{C}_1$ and $\mathcal{K}$.
Defining $g$ thanks to Lemma \ref{lem:g},  we have the following boundary conditions:
\begin{align}
w_2(t,x_1(t))&=\omega(t)\,,& w_1(t,x_2(t))=g(t,x_2(t),w_2(t,x_2(t)))\,, 
\end{align}
where $\omega$ is chosen so that $w_2$ is $\C1$ through $\mathcal{C}_1$. The definition of $g$ ensures that the compatibility condition given by the second equation of (\ref{eq:cond_saut}) is satisfied along the discontinuity line $\mathcal{K}$.

\begin{lemma}\label{lem:derivee}
Let us assume that, along the shock $\mathcal{K}$, the conditions (\ref{eq:cond_saut}) are satisfied.  Then
along the shock $\mathcal{K}$, we have
\begin{align*}
\pt_r w_1^-(t,x_2(t)) =&\frac{1}{U-\lambda_1^-}\left[(U-\lambda_2^-)\pt_{w_2^-}g\, \pt_r w_2^- - (\pt_{w_2^-}g +1) f^-\right.\\
&\left. +(U-\lambda_1^+)\pt_{w_1^+}g \,\pt_r w_1^+ +(U-\lambda_2^+)\pt_{w_2^+} g \, \pt_r w_2^+ +(\pt_{w_1^+}g - \pt_{w_2^+} g )f^+\right]\,.
\end{align*}
\end{lemma}

\begin{proof}
Derivating $w_1^-(t, x_2(t))$ we obtain
\begin{align*}
&\frac{\d{}}{\d{t}}w_1^-(t, x_2(t))\\
=&\pt_{w_2^-}g((U-\lambda_2^-)\pt_r w_2^- -f^-)+\pt_{w_1^+}g((U-\lambda_1^+)\pt_r w_1^+ +f^+) +\pt_{w_2^+}g((U-\lambda_2^+)\pt_r w_2^+ -f^+)
\end{align*}
but also $\frac{\d{}}{\d{t}}w_1^-(t, x_2(t)) =\pt_t w_1^-+U\pt_r w_1^-
=(U-\lambda_1^-)\pt_r w_1^- +f^-$. 
\end{proof}

\begin{proposition} \label{prop:constr}
Let $R_0>0$. Let us consider a Bethe-Weyl gas, satisfying $1< \mathscr{G}< 2$ and such that $(\rho \mapsto \frac{c(\rho)}{\rho})$ is integrable in 0.

Let us consider the free boundary problem consisting in the system (\ref{eq:source}) in $D_0$ with the following boundary conditions:
\begin{align*}
w_2(t, x_1(t))& =\omega(t)\geq 0 \,,& \textrm{on } \mathcal{C}_1\,;& &w_1(t, x_2(t))&=G(t,w_2)>0\,,&\textrm{ on } \mathcal{K}\,,
\end{align*}
with $\mathcal{C}_1$, $\mathcal{K}$, defined as above. 

Assume  that the compatibility conditions (\ref{eq:cond_saut}) are satisfied at time $t=0$ and that there exists $C_0>0$ such that, for any $M, T>0$:
\begin{align*}
\textrm{along } \mathcal{C}_1\,,&\quad\modulo{\pt_rw_2(t,x_1(t))}=\modulo{\frac{\omega' +f}{\lambda_1-\lambda_2}(t,x_1(t))}\leq \frac{C_0}{x_1(t)} \,,\\
\textrm{along } \mathcal{K}\,,&\quad\pt_t G(t,x_2(t))\geq \frac{-C_0}{x_2(t)}\,,\qquad\qquad \sup_{ t\in [0,T]}\modulo{\pt_2 G(t,x_2(t))} \leq C_0\,.
\end{align*}

Then, assuming that there exists $\delta>0$ such that $\rho^--\rho^+\geq \delta>0$ along $\mathcal{K}$, there exists a $\C1$ solution of (\ref{eq:source}) in $D_0$ whose time of existence is bounded   below in the following way
\begin{align*}
T_{ex}\geq C R_0\,,
\end{align*}
where $C$ depends on the boundary conditions.
\end{proposition}

\begin{proof}
Let $(t,r)\in D_0$.  The 2-characteristic going through $(t,r)$ originates in $D_-$  and crosses $\mathcal{C}_1$ in $(\beta, x_1(\beta))$; meanwhile the 1-characteristic going through $(t,r)$ originates in $(\alpha, x_2(\alpha))\in \mathcal{K}$.

As $w_1> 0$ along $\mathcal{K}$, we can apply Lemma  \ref{lem:C0} on $ D_0$ to obtain a $\C0$ estimate of $w_1$ and $w_2$. We obtain
\[
0\leq w_1(\alpha, x_2(\alpha) ) \leq w_1(t,r) \leq w_2(t,r)  \leq w_2(\beta, x_1(\beta)) \leq \norma{\omega}_{\L\infty}\,.
\]

Next, we want to obtain $\C1$ estimates for $w_1$ and $w_2$. 
First, we derive an $\L\infty$ estimate for  $\pt_r w_2$. Let us remind that $v_2=e^h(\pt_r w_2+\Psi)$ with $\Psi (r,w)=\frac{\psi(w)}{r}$. 
\begin{itemize}
\item In the case $(\pt_r w_2 +\Psi)(\beta, x_1(\beta))\geq 0$, then we have $v_2(\beta,x_1(\beta))\geq 0$ and we can apply Lemma \ref{lem:w'pos} on $D_0$ to  obtain an $\L\infty$-estimate of $\pt_r w_2$    in $D_0$. Since $X_2'\geq 0$,  we obtain a lower bound for the time of existence of $\pt_r w_2$: 
\[
T_{ex}\geq   x_1(\beta) \frac{x(B_0, B_1, B_2)}{B_1} \geq  R_0 \frac{x(B_0, B_1, B_2)}{B_1} \,,
\]
where $B_0$, $B_1$, $B_2$ are defined as in (\ref{eq:B0B1B2}).
\item In the case $v_2(\beta,x_1(\beta))\leq 0$, as $X_2'\geq 0$, by Lemma \ref{lem:w'neg} we  obtain the estimate
\[
T_{ex}\geq \frac{R_0} {B_1} Q^{-1}\left(\Xi_+\right)
\]
where $B_1$  is  defined as in (\ref{eq:B0B1B2}) and $\Xi_+ = \frac{B_1}{\sqrt{B_0 B_2}} \frac{1}{1+\frac {\sqrt{B_0}} {\sqrt{B_2} } ( \modulo{v_2(\beta,x_1(\beta))}  x_1(\beta))} $. 

Since $\Psi(w,r)=\frac{\psi(w)}{r}\leq \frac{\norma{\psi}_{\L\infty}}{r}$ and, by hypothesis $\modulo{\pt_r w_2(\beta, x_1(\beta))}\leq \frac{C_0}{x_1(\beta)}$, we have 
\[
\modulo{v_2(\beta,x_1(\beta))}  x_1(\beta)\leq \norma{\psi}_{\L\infty}+C_0 :=C_1\,,
\]
and we finally have
\[
T_{ex}\geq \frac{R_0} {B_1}Q^{-1}\left(\frac{B_1}{\sqrt{B_0 B_2}} \frac{1}{1+\frac {\sqrt{B_0}} {\sqrt{B_2} } C_1} \right)\,.
\]
\end{itemize}

Let us now find an estimate in $\L\infty$ for $\pt_r w_1$. We want to proceed in the same way. First we prove that there exists $C$ such that  $\pt_r w_1+\Phi\geq \frac{-C}{r}$ along $\mathcal{K}$. derivating the boundary condition, we get
\begin{align*}
\frac{\d{}}{\d{t}}w_1(t, x_2(t))&=
\pt_t w_1+ U\pt_r w_1\\
&= f+(U-\lambda_1)\pt_r w_1\\
&=\pt_t G+\pt_2 G\left((U-\lambda_2)\pt_r w_2 -f\right)\,.
\end{align*}
Hence $\pt_r w_1=\frac{1}{U-\lambda_1}\left[ \pt_t G-f +\pt_2 G\left((U-\lambda_2)\pt_r w_2 -f\right)\right]$.
By hypothesis we have   a lower bound on $\pt_t G$. We can find a similar lower bound on $f$, whose expression is known. By hypothesis $U-\lambda_2^-\leq 0$, $U-\lambda_1^- >0$ and they are  depending only on $(w_1, w_2)$. Thanks to Lemma \ref{lem:up}, we obtain the following upper bound on $v_2=e^h(\pt_r w_2+\Psi)$,  reminding that, by Lemma \ref{lem:compab}, $b_i(r,w)=\bar b_i(w)/r^i$ for any  $i\in \{0,1,2\}$:
\begin{align*}
v_2(t, X_2(t))e^{-\int_\alpha^t \frac{(\lambda_2+\bar b_1)(s, X_2(s))}{X_2(s)}\d{s}}\leq \frac{X_2(\alpha) v_2(\alpha, X_1(\alpha))}{X_2(t)} +\frac{1}{X_2(t)} \int_\alpha^t \frac{\bar b_2}{X_2(s)} e^{-\int_\alpha^s\frac{(\lambda_2+\bar b_1)(\tau, X_2(\tau))}{X_2(\tau)} \d{\tau}}\d{s}\,.
\end{align*}  
Reminding that $X_1'\geq 0$ and $X_2'\geq 0$, if $(\alpha, X_2(\alpha))\in \mathcal{C}_1$, we have, for any $s\geq \alpha$, $1/X_2(s)\leq 1/X_2(\alpha)\leq 1/R_0$. Using furthermore the hypothesis $\modulo{r(\pt_r w_2 + \Psi)}\leq C_0$ along $\mathcal{C}_1$, we obtain
 \begin{align*}
v_2(t, X_2(t))\leq \left(\frac{C_0}{X_2(t)}+\frac{1}{X_2(t)} (t-\alpha) \frac{B_2}{R_0}\right) e^{(t-\alpha) (\sup(\lambda_2)+B_1)/R_0}\,,
\end{align*} 
which provide us the estimate $v_2(t,x_2(t)) \leq \frac{1}{x_2(t)}C e^{\frac{C (t-\alpha)}{R_0}}$. We obtain the desired estimate taking a time $t$ bounded.
\begin{itemize}
\item 
If $(\pt_r w_1 +\Phi)(\alpha, x_2(\alpha))\geq 0$, then we have $v_1(\alpha,x_2(\alpha))\geq 0$ and we can apply Lemma \ref{lem:w'pos} on $D_0$ to  obtain estimate of $\pt_r w_1$  in $\L\infty$ in $D_0$. Since $X_1'\geq 0$,  we obtain a lower bound for the time of existence of $\pt_r w_1$: 
\[
T_{ex}\geq   x_2(\alpha) \frac{x(A_0, A_1, A_2)}{A_1} \geq  R_0 \frac{x(A_0, A_1, A_2)}{A_1} \,,
\]
where $A_0$, $A_1$, $A_2$ are defined as in (\ref{eq:A0A1A2}) and $x(A_0, A_1, A_2)$ is defined in Lemma \ref{lem:w'pos}.
\item  If $v_1(\alpha,x_2(\alpha))\leq 0$, as $X_1'\geq 0$, by Lemma \ref{lem:w'neg} we  obtain the estimate
\[
T_{ex}\geq \frac{R_0} {A_1} Q^{-1}\left(\Theta_+\right)
\]
where $A_1$  is  defined as in (\ref{eq:A0A1A2}) and $\Theta_+ = \frac{A_1}{\sqrt{A_0 A_2}} \frac{1}{1+\frac {\sqrt{A_0}} {\sqrt{A_2} } ( \modulo{v_1(\alpha,x_2(\alpha))}  x_2(\alpha))} $. Since $\Phi(w,r)=\frac{\phi(w)}{r}\leq \frac{C_1}{r}$ by Lemma \ref{lem:compab} and $\modulo{\pt_r w_1(\alpha, x_2(\alpha))}\leq \frac{C_0}{x_2(\alpha)}$, we have 
\[
\modulo{v_1(\alpha,x_2(\alpha))}  x_2(\alpha)\leq C_1+C_0 :=C_2\,,
\]
and we finally have
\[
T_{ex}\geq \frac{R_0} {A_1} Q^{-1}\left(\frac{A_1}{\sqrt{A_0 A_2}} \frac{1}{1+\frac {\sqrt{A_0}} {\sqrt{A_2} } C_2} \right)\,.
\]
\end{itemize}
\end{proof}

We are now able to construct a piecewise solution:
\begin{theorem}\label{thm:main}
Let us consider a Bethe-Weyl gas satisfying $1<\mathscr{G}<2$ and such that $(\rho \mapsto \frac{c(\rho)}{\rho})$ is integrable in 0.

Let $R_0>0$. Assume there exists two regular solutions of (\ref{eq:source}) $(\rho^-, u^-)$ and $(\rho^+, u^+)$ with a time of existence $\mathcal{T}>0$, $(\rho^-, u^-)$ admitting an initial condition at time $t=0$ defined on $[0, R_0]$ and $(\rho^+, u^+)$ admitting an initial condition  at time $t=0$ on $[R_0, +\infty[$. We assume that the compatibility condition (\ref{eq:cond_saut})  are satisfied at time $t=0$, in $r=R_0$. We denote  $w=(w_1, w_2)$  the associated Riemann invariant defined as in (\ref{def:w}). The superspcript $\pm$ means the quantity asoociated to $(\rho^\pm,u^\pm)$ definied in the domains $D^\pm$. The subscript 0 indicates the initial condition at time $t=0$.

Let us assume that
\begin{align*}
 \min_{r\leq R_0}w_{1,0}^-(r)&> 0 \,, & w_{1,0}^-(R_0)+\inf_{r\geq R_0} w_{1,0}^+(r)&>0\,.
\end{align*}
and  that there exists $C_0>0$ such that, for all $t\in [0, \mathcal{T}[$ and any $r$ where $w^-$ (respectively $w^+$) is defined:
\begin{align*}
\modulo{\pt_r w_2^-(t, r)}&\leq \frac{C_0}{r}\,,  \pt_r w_1^+(t,r) &\geq \frac{-C_0}{r}\,,&  \pt_r w_2^+(t,r) &\leq \frac{C_0}{r}\,.
\end{align*}

Then, as long as $\rho^->\rho^+$ along $\mathcal{K}$, there exists a shock wave solution of (\ref{eq:source}) with initial conditions
\begin{align*}
\textrm{for }&r\leq R_0\,,& u(0,r)&=u_0^-(r)\,,&\rho(0, r)&=\rho_0^-(r)\,,\\
\textrm{for }&r\geq R_0\,,& u(0,r)&=u_0^+(r)\,,&\rho(0,r) &=\rho_0^+(r)\,,
\end{align*}
and its time of existence is bounded below in the following way
\begin{align*}
T_{ex}\geq \min(\mathcal{T}, C R_0)\,,
\end{align*}
where $C$ depends on the initial conditions.
\end{theorem} 

For example, we can use the result obtained in \cite{mercier_solreg} in order to have $\mathcal{T}=+\infty$.
For $(\rho^+, u^+)$, a stationary solution could also work.

\begin{proof}
Let $D_-, D_0, D_+$ be defined as before.
Let us define $G(t, w_2):=g(w_2, w^+)$, where $g$ has been defined in Lemma \ref{lem:g}. Using the estimate of Lemma  \ref{lem:w1posK} in $D_+$, we obtain along the discontinuity line $\mathcal{K}$: $w_1^-=g(w_2^-, w^+)\geq w_1^+\geq w_{1,0}^+\geq \inf_{ \{r\geq R_0\} } w_{1,0}^+$. 
Besides, $u> 0$ in $D_-$, in particular $u^-(0,R_0)>0$, and by continuity of $u^-$ in $D_-\cup D_0$, at least for a small time, $u^-$ is non-negative in $D_0$. 
Note $[0,t_0]$ the time interval in which $u^-$  is non-negative in $D_0$. For $(t,r)\in D_0$ such that $t\leq t_0$. Let us denote $\beta$ the time at which the 2-characteristic going through $(t,r)$ cross $\mathcal{C}_1$.
\[
u^-(t,r)\geq \frac{1}{2}(w_1^-+w_2^-)(t,x_2(t))\geq \frac{1}{2}(\min_{r\geq R_0} (w_{1,0}^+(r)) + \frac{w_{2}(\beta, x_1(\beta))}{1+w_2(\beta, x_1(\beta) )\frac{(d-1)(t-\beta)}{4x_1(\beta)})} )\,.
\]
Besides, as $(\beta,x_1(\beta))\in D_-$, we have $w_2(\beta, x_1(\beta))\geq w_1(\beta, x_1(\beta))\geq w_1(0, R_0)$. 
Furthermore, $x_1'\geq 0$ then $x_1(\beta)\geq R_0$. Then
\[
u^(t,r)\geq \frac{1}{2}(w_1^-+w_2^-)(t,x_2(t))\geq \frac{1}{2}(\min_{r\geq R_0} (w_{1,0}^+) + \frac{1}{\frac{1}{w_1^-(0,R_0 )}+\frac{(d-1)(t-\beta)}{4R_0})} )
\]
If $w_{1,0}^+\geq 0$, $u^-\geq 0$ in all $D_0$. Otherwise, if $w_{1,0}^+<0$ then $u^-$ is non-negative  as long as 
\[
t-\beta \leq \frac{4 R_0}{(d-1)}\left(\frac{1}{\min(w_{1,0}^+)}+\frac{1}{w_2(\beta, x_1(\beta))}\right)\,.
\] 
Then we obtain that $u^-$ is non-negative if 
\[
t\leq  \frac{4 R_0}{(d-1)}\left(\frac{1}{\min(w_{1,0}^+)}+\frac{1}{w_1^-(0,R_0)}\right)\,.
\]
Hence, we can apply Proposition \ref{prop:CO2} to $(t,r)\in D_0$ to find a time of validity for the $\C0$-estimates proportional to $R_0$.

Besides, the expression of $\pt_r w_1^-$ along $\mathcal{K}$ in Lemma \ref{lem:derivee} and the hypotheses on $\pt_r w_1^+$, $\pt_r w_2^+$ allow us to check that the hypotheses of Proposition \ref{prop:constr} are satisfied. Hence, we have a regular solution in the angular domain satisfying the boundary condition $w_2(t, x_1(t))=w_2^-(t, x_1(t))$ and $w_1(t, x_2(t))=g(w_2, w^+)$.

\end{proof}

\subsection{Checking hypotheses}
We determine now a set of hypotheses on the initial conditions so that the hypothesis ``$\rho^->\rho^+$'' of the previous theorem is satisfied. The condition $\modulo{\pt_r w_2}\leq \frac{C}{r}$ along $\mathcal{C}_1$ can not be computed with this method: the construction of a regular solution in $D_-$ has to be obtained by another result on regular solution (see for example D. Serre \cite{Sr97} or M. Grassin \cite{Grass} or M. Lécureux-Mercier \cite{mercier_solreg}).
\begin{proposition}\label{prop:condD}
%
%
%

In the same context as in Theorem \ref{thm:main}, we assume furthermore that $w_{1,0}^-(R_0)+\min_{ \{r\geq R_0\} } w_{1,0}^+>0$ and $w_{1,0}^-(R_0)>\max_{\{r\geq R_0\}} w_{2,0}^+$. Then, the estimates of Proposition \ref{prop:CO2} are available in $D_0$ for a time proportional to $R_0$ and furthermore, along $\mathcal{K}$, $\rho^+<\rho^-$.
\end{proposition}


\begin{proof}
Let $t\in\reali_+$ we denote $\beta$ the time at which the 2-characteristic going through $(t,x_2(t))$ crosses $\mathcal{C}_1$. By Proposition \ref{prop:CO2}, we have
\begin{align*}
w_2^-(t,x_2(t))\geq& \frac{w_2(\beta, x_1(\beta))}{1+(t-\beta)\frac{(d-1)}{4 x_1(\beta)} w_2(\beta, x_1(\beta)) }\\
\geq & \frac{1}{\frac{1}{w_{1,0}^-(R_0)} +(t-\beta)\frac{(d-1)}{4R_0}}\,.
\end{align*}
Besides, $w_2^+(t,x_2(t))\leq \max (w_{2,0}^+)$.
By remark \ref{rem:condDens}, the condition $\rho^+<\rho^-$ on $\mathcal{K}$ is equivalent to $w_2^->w_2^+$. Hence it is sufficient to have 
\[
\frac{1}{\frac{1}{w_{1,0}^-(R_0)} +(t-\beta)\frac{(d-1)}{4R_0}}> \max (w_{2,0}^+)
\]
Finally it is sufficient to have $w_{1,0}^-(R_0)>\max(w_{2,0}^+)$ and $t< \frac{4 R_0}{(d-1)}\frac{w_1^-(0,R_0)-\max w_{2,0}^+}{w_{1}^-(0,R_0)\max w_{2,0}^+}$.
\end{proof}

\appendix
\section{Time of existence for ODE}\label{sec:ode}
\begin{lemma}[Maximum principle for ODE.]\label{lem:ppemax}
Let $a,b:[0,T]\to \reali$  be continuous applications. Assume that  $w,z:[0,T]\to \reali$ are continuous applications such that $z(0)\geq w(0)$ and $w'=aw^2+b$, $z'\geq az^2+b$ in $[0,T]$. Then 
$z(t)\geq w(t)$ for all $t\in [0,T]$.
\end{lemma}
\begin{proof}
By hypothesis, 
\begin{align*}
 (z-w)'(t)&=z'-(a(t)w(t)^2+b(t))\geq a(t)(z(t)^2-w(t)^2)\,,
\end{align*}
that is to say $(z-w)'(t)\geq a(t)(z(t)-w(t))(z(t)+w(t))$. It follows that, 
\[
\frac{\d{}}{\d{t}}\left((z-w)(t)e^{-\int_0^t a(s) (z+w)(s)\d{s}}\right)\geq 0\,.
\]
Hence $(z-w)(t) e^{-\int_0^t a(s) (z+w)(s)\d{s}}\geq (z-w)(0)\geq 0$, for all $t\in[0,T]$.
\end{proof}

\begin{lemma}\label{lem:hor2}
Let $T>0$ and   $a_0, a_1, a_2\in \C0([0,T];\reali)$. Let $a_0^+=\max(a_0,0)$ and $K$  be defined by 
\begin{align}
K&= \int_0^T \modulo{a_2(t)}\d{t} \exp\left(\int_0^T \modulo{a_1(t)}\d{t}\right)\,. \label{def:k}
\end{align}
If $y_0\geq 0$ and 
\begin{align}
\frac{1}{y_0+K}&>\int_0^T {a_0^+(t)}\d{t}\exp\left(\int_0^T \modulo{a_1(t)}\d{t}\right)\,,\label{hyp:1} \\
\frac{1}{K}&>\int_0^T\modulo{a_0(t)}\d{t}\exp\left(\int_0^T \modulo{a_1(t)}\d{t}\right)\,,\label{hyp:2}
\end{align}
then the maximal solution of the Cauchy problem
\begin{align}
y' &=a_0(t)y^2 +a_1(t)y +a_2(t)\,,&y(0)&=y_0\,, \label{ann:edo}
\end{align}
 is defined at least on $[0,T]$ and satisfies
\begin{align}
\frac{1}{y(T)}&>\frac{1}{y_0+K}-\int_0^T\modulo{a_0^+(t)}\d{t}\exp\left(\int_0^T \modulo{a_1(t)}\d{t}\right)\,,&\textrm{if }& y(T)\geq 0\,,\label{ineg:pos}\\
\frac{1}{\modulo{y(T)}}&>\frac{1}{K}-\int_0^T\modulo{a_0(t)}\d{t}\exp\left(\int_0^T \modulo{a_1(t)}\d{t}\right)\,,&\textrm{if }& y(T)< 0\,.\label{ine:neg}
\end{align}
\end{lemma}

\begin{proof}
First, denoting $\tilde y=\exp\left({-\int_0^t a_1(s)\d{s}}\right) y$, $\tilde a_0=\exp\left({\int_0^t a_1(s)\d{s}}\right) a_0 $ and   $\tilde a_2 =\exp\left({-\int_0^t a_1(s)\d{s}}\right) a_2$,  we see that the equation (\ref{ann:edo})  becomes the ordinary differential equation
\begin{equation}\label{edotilde}
{\tilde y}' =\tilde a_0(t){\tilde y}^2+\tilde a_2(t)\,.
\end{equation}
and $\tilde y(0)=y(0)$. We can thus assume without loss of generality that $a_1\equiv 0$. 

Let us introduce the increasing function $v$  defined by  $v(t)=\int_0^t \modulo{\tilde a_2(s)}\d{s}$. Let $z$ be the maximal solution of the Cauchy problem
 \begin{align*}
 z'&={\tilde a}_0^+(t)(z+K)^2\,,&z(0)&=y_0\,.
 \end{align*}
Then $z$ is increasing and since $y_0+K>0$, we have
 \[
 \frac{1}{z(t)+K}=\frac{1}{y_0+K} -\int_0^t {\tilde a}_0^+(s)\d{s}\,.
 \]

Note that the right hand side does not vanish for all  $t\in [0,T]$, thanks to (\ref{hyp:1}). Besides, we have  $(z+v)(0)=y_0$ and 
 \begin{align*}
 (z+v)' &={\tilde a}_0^+(t)(z+K)^2+\modulo{\tilde a_2}\\
&\geq {\tilde a}_0^+(t)(z+v)^2+{\tilde a_2}\\
&\geq \tilde a_0(t) (z+v)^2 +\tilde a_2\,.
 \end{align*}
 Consequently, according to  Lemma \ref{lem:ppemax}, we have $y(t)\leq (z+v)(t)\leq z(t)+K$ for all $t\in [0,T]$ if $y$ exists. In particular, as long as   $y(t)>0$, we have
 \[
\frac{1}{y(t)}\geq\frac{1}{z(t)+K}= \frac{1}{y_0+K} -\int_0^t {\tilde a}_0^+(s)\d{s}=\frac{1}{y_0+K} -\int_0^t { a}_0^+(\tau)e^{\int_0^\tau a_1(s)\d{s}}\d{\tau}\,,
 \]
hence
\begin{equation}\label{wpos}
\frac{1}{y(t)}\geq\frac{1}{y_0+K} -\int_0^t { a}_0^+(\tau)\d{\tau}\,e^{\int_0^t \modulo{ a_1(s)}\d{s}}\,.
\end{equation}

Assume now that $y$ vanishes and changes its sign in  $t_0\in [0,T]$.  We can apply the same procedure as above to  $Y=-y$, replacing $y_0$ by $Y(t_0)=Y_0=0$ and beginning at time  $t_0$. The application $Y$ is then solution of $Y'=A_0 Y^2+a_1 Y+A_2$, where $A_0=-a_0$, $A_2=-a_2$. Denoting $A_0^+=\max(-a_0,0)$, we get that, for all $t\geq t_0$ such that  $Y(t)>0$ 
\[
\frac{1}{Y(t)}\geq\frac{1}{K} -\int_{t_0}^t { A}_0^+(\tau)\d{\tau}\,e^{\int_0^t \modulo{ a_1(s)}\d{s}}\,.
\]
Consequently, for all $t$ such that $y(t)<0$
\begin{equation}\label{wneg}
\frac{1}{\modulo{y(t)} }\geq\frac{1}{K} -\int_0^t \modulo{a_0(\tau)}\d{\tau}\, e^{\int_0^t \modulo{ a_1(s)}\d{s}}\,.
\end{equation}

Finally, the inequalities  (\ref{wpos})--(\ref{wneg}) give us some bounds on  $y$  for all time, $y$ being positive or negative. Hence, $y$ can not tend to $\pm\infty$ and exists up to time $T$.
Indeed, we proved above that there exists a function $\phi\in \C0([0,T],\reali)$ such that, if $y$ is solution of (\ref{ann:edo}), then  $\modulo{y(t)}\leq \phi(t)$  for all  $t\in [0,T]$. 
Let us denote $T_*$ the maximal time of existence of  $y$. If we assume $T_*<T$, then we obtain that $y$ is bounded on $[0,T_*[$ by $\max_{[0,T]}\phi$, which contradicts the fact that $y$ has to go out of all compact set when $t\to T_*<\infty$.
\end{proof}

\section{Explicit expression of the coefficients}\label{sec:vdw}

We use here the same notation as in Lemma \ref{lem:compab}.
\begin{remark}
Note  that, in the case $\rho\mapsto \frac{\sqrt{H'(\rho)}}{\rho}$ is integrable at $+\infty$, we have 
\begin{align*}
A&=1+g=\frac{1}{2\sqrt{H'}}\int_{+\infty}^\rho \frac{1}{u}\sqrt{H'(u)}(\mathscr{G}-1)\d{u}\leq 0\,.\\
\end{align*}
Hence $1+2g\leq 0$ and we have also $B\leq 0$. Besides,  
\begin{align*}
(1+g)H-B&=\frac{1}{2\sqrt{H'}}\int_0^\rho \left(H'\sqrt{H'}+\frac{H\sqrt{H'}}{\rho} (\mathscr{G}-1)\right)\geq 0\,.
\end{align*}
Hence, if $w_2=u+H\geq 0$, then
$
\Psi \leq\frac{(d-1)}{r}(1+g)(u+H)\leq 0
$
\end{remark}

\begin{remark}
In the case $d=3$, we obtain
\begin{align*}
a_2&=\frac{2\sqrt{H'}}{r^2}\big[
-u^2 \left(\mathscr{G}A^2 -2A (\mathscr{G}-2)+\mathscr{G}-1\right)
+u\left(2(A-1) (\mathscr{G}B+c) +4B\right)-\mathscr{G}B^2
\big]\\
&=\frac{2\sqrt{H'}}{r^2}\big[
-\mathscr{G}(u(A-1+\frac{2}{\mathscr{G}})-B)^2 -4 u^2(\frac{3}{4}-\frac{1}{\mathscr{G}}) +2uc(A-1)
\big]\,,
\end{align*}
and
\begin{align*}
b_2&= \frac{2\sqrt{H'}}{r^2}\big[
-u^2\left(\mathscr{G}A^2-2A(\mathscr{G}-2) +\mathscr{G}-1\right) 
+u\left(2B(\mathscr{G}-2-\mathscr{G}A) -2c(A-1)\right) -\mathscr{G}B^2
\big]
\end{align*}
In the case $d=2$, we obtain
\begin{align*}
a_2&=\frac{\sqrt{H'}}{2r^2}\big[-\mathscr{G} \left(u(A-1+\frac{5}{2\mathscr{G}} ) -B\right)^2 -\frac{25}{4} u^2\left(\frac{16}{25} -\frac{1}{\mathscr{G}}\right) +3uc(A-1)\big]\,.
\end{align*}

\end{remark}
\subsection{Perfect gas.}
For a perfect gas, we have: $c(\rho)=\sqrt{\g(\g-1)}\left(\rho\right)^{\frac{\g-1}{2}}$, thus
\begin{align*}
H&=2\sqrt{ \frac{\g}{\g-1}} \left( \rho\right)^{\frac{\g-1}{2}}\,.
\end{align*}
This implies, denoting $\nu=\frac{ \g+1 }{\g-1}>1$
\begin{align*}
\rho &= (\frac{\g-1}{4\g})^{\frac{1}{\g-1}} H^{\nu-1} \,,\\
c(\rho)&=\frac{\g-1}{2} H\,,\\
H'&=\frac{\g-1}{2} (\frac{4\g}{\g-1})^{\frac{1}{\g-1}}  H^{2-\nu}
\end{align*}
Consequently, noting that $\frac{\g-1}{2}=\frac{1}{\nu-1}$, $u=\frac{w_1+w_2}{2}$ and $H=\frac{w_2-w_1}{2}$, we get 
\begin{align*}
\lambda_1&=\frac{w_1+w_2}{2}-\frac{1}{\nu-1}\left(\frac{w_2-w_1}{2}\right)\,,\\
\lambda_2&=\frac{w_1+w_2}{2}+\frac{1}{\nu-1}\left(\frac{w_2-w_1}{2}\right)\,,\\ 
\pt_1\lambda_1& =\pt_2\lambda_2  =\frac{\mathscr{G}}{2}=\frac{\e}{2(\e-1)} >0\,,\\
\pt_2\lambda_1& =\pt_1\lambda_2=\frac{(\e-2) }{2(\e-1)}\geq 0\,.
\end{align*}
Then $h=k=\mathrm{ln} \left( H^{ (2-\nu)/2}\right)$ and 
\begin{align*}
g&=\frac{\nu-1}{(2-\nu) }\leq 0\,,\\
B&=\frac{\nu}{(2-\nu)(4-\nu)} H  \,,\\
a_0&=\frac{-\nu}{2(\nu-1)^{1/2}(2\nu+2)^{\frac{\nu-1}{2}}} H^{\nu-1}\,,\\
a_1&=\frac{d-1}{(2-\nu)r}\big[ 3u+\frac{2(4-3\nu)H}{(\nu-1)(4-\nu)} \big]\,,\\
a_2&= \frac{d-1}{(2-\nu)r^2}\big[
-u^2 +\frac{4uH}{(4-\nu)} - \frac{\nu H^2}{(\nu-1)(4-\nu)} \\
&\quad +\frac{d-1}{2}\left( \frac{-2u^2}{(2-\nu)} -\frac{8uH}{(2-\nu)(4-\nu)} +( \frac{\nu H^2}{(\nu-1)(4-\nu)}-\frac{\nu^3}{(\nu-1)(2-\nu)(4-\nu)^2} )H^2 \right)
\big]
\end{align*}

\subsection{Van der Waals gas.}
For a Van der Waals gas, we have: $c(\rho)=\frac{1}{1-b\rho}\sqrt{\g(\g-1)}\left(\frac{\rho}{1-b\rho}\right)^{\frac{\g-1}{2}}$, thus
\begin{align*}
H&=2\sqrt{ \frac{\g}{\g-1}} \left( \frac{\rho}{1-b\rho}\right)^{\frac{\g-1}{2}}\,.
\end{align*}
This implies, denoting $\tilde b=b(\frac{\g-1}{4\g} )^{\frac{1}{\g-1}}$ and $\nu=\frac{ \g+1 }{\g-1}>1$
\begin{align*}
\rho&=\frac{ (\frac{\g-1}{4\g})^{\frac{1}{\g-1}} H^{\nu-1} }{ 1+\tilde b H^{\nu-1}}\,,\\
c(\rho)&=\frac{\g-1}{2}(1+\tilde b H^{\nu-1}) H\,,\\
H'&=\frac{\g-1}{2} (\frac{4\g}{\g-1})^{\frac{1}{\g-1}} (1+\tilde b H^{\nu-1})^2 H^{2-\nu}
\end{align*}
Consequently, noting that $\frac{\g-1}{2}=\frac{1}{\nu-1}$, $u=\frac{w_1+w_2}{2}$ and $H=\frac{w_2-w_1}{2}$, we get 
\begin{align*}
\lambda_1&=\frac{w_1+w_2}{2}-\frac{1}{\nu-1}\left(\frac{w_2-w_1}{2}+ \tilde b(\frac{w_2-w_1}{2})^{\nu}\right)\,,\\
\lambda_2&=\frac{w_1+w_2}{2}+\frac{1}{\nu-1}\left(\frac{w_2-w_1}{2}+ \tilde b(\frac{w_2-w_1}{2})^{\nu}\right)\,,\\ \pt_1\lambda_1&=\pt_2\lambda_2  =\frac{\mathscr{G}}{2}=\frac{\e}{2(\e-1)} (1+\tilde b H^{\nu-1})>0\,,\\
\pt_2\lambda_1&=\pt_1\lambda_2=\frac{1}{2(\e-1)}\left((\e-2) -\tilde b \e H^{ \nu-1 }\right)\,.
\end{align*}
Then $h=k=\mathrm{ln} \left( H^{ (2-\nu)/2} (1+\tilde b H^{\nu-1})\right)$ and 
\begin{align*}
g&=\frac{\nu-1}{(2-\nu)(1+\tilde{b} H^{\nu-1}) }\,,\\
B&=\frac{1}{(2-\nu)(4-\nu)}\frac{H}{1+\tilde{b}H^{\nu-1}} \left(1+\frac{(2-\nu)(4-\nu)}{2+\nu}\tilde{b} H^{\nu-1}\right)\,,\\
a_0&=\frac{-\nu}{2(\nu-1)^{1/2}(2\nu+2)^{\frac{\nu-1}{2}}} H^{\nu-1}\,.
\end{align*}

\small{

 \bibliography{ref}

 \bibliographystyle{abbrv} }

\end{document}